\documentclass[11pt]{article}

\usepackage{amsthm,amsmath,amssymb}

\usepackage{graphicx,url,enumerate}
\usepackage{hyperref}
\usepackage[utf8]{inputenc}
\usepackage[msc-links]{amsrefs}
\usepackage{xspace}
\usepackage[charter]{mathdesign}
\theoremstyle{remark}
\newtheorem{rema}{Remark}

\theoremstyle{definition}

\theoremstyle{plain}
\newtheorem{theo}{Theorem}
\newtheorem{coro}[theo]{Corollary}
\newtheorem{lemm}[theo]{Lemma}
\newtheorem{prop}[theo]{Proposition}

\numberwithin{equation}{section}
\numberwithin{theo}{section}

\usepackage{fancybox}
{\par\begin{Sbox}\begin{minipage}{\textwidth}}%
{\end{minipage}\end{Sbox}\par\smallskip\shadowbox{\TheSbox}}

\long\def\xcom#1{}

\newcommand{\R}{\mathbb{R}}
\newcommand{\PP}{\mathbb{P}}
\newcommand{\E}{\mathbb{E}}
\newcommand{\N}{\mathbb{N}}
\newcommand{\Q}{\mathbb{Q}}
\newcommand{\Z}{\mathbb{Z}}

\newcommand{\ch}{{\mathop{\rm ch}}}

\newcommand{\unsur}[1]{{\frac{1}{#1}}}
\def\un#1{{{\,\mathbf{1}}_{({#1})}}}
\def\valabs#1{{\left\vert {#1} \right \vert}}
\def\esp#1{{{\E}\left [ {#1} \right ]}}
\def\etp#1{{\left ( {#1} \right )}}
\def\etc#1{{\left [ {#1} \right ]}}

\def\esperance#1#2{{{\E}_{#1}\left [ {#2} \right ]}}
\def\prob#1{{{\PP}\left ( {#1} \right )}}

\def\ens#1{{\left\{#1\right\}}}
\newcommand{\floor}[1]{\lfloor#1\rfloor}

\newcommand{\egaloi}{{\,\mathrel{\mathop{=}\limits^d}\,}}

\def\dessus#1#2{\mathord{\mathop{\kern 0pt #2}\limits^#1}}

\newcommand{\bit}{\begin{itemize}}
\newcommand{\eit}{\end{itemize}}
\newcommand{\ben}{\begin{enumerate}}
\newcommand{\een}{\end{enumerate}}

\newcounter{moncompteur}

\newenvironment{myenumerate}%
{\begin{list}{\arabic{moncompteur}. }{\usecounter{moncompteur}%
\setlength{\leftmargin}{0pt}%
\setlength{\labelwidth}{0pt}%
\setlength{\listparindent}{0pt}%
\setlength{\labelsep}{0pt}}}%
{\end{list}}

\def\bmen{\begin{myenumerate}}
\def\emen{\end{myenumerate}}

\newcommand{\undemi}{\frac{1}{2}}

\newcommand{\Arond}{{\mathcal A}}

\newcommand{\Crond}{{\mathcal C}}
\newcommand{\Drond}{{\mathcal D}}

\newcommand{\Frond}{{\mathcal F}}

\newcommand{\Srond}{{\mathcal S}}

\newcommand{\e}{{\mathbb E}}
\newcommand{\p}{{\mathbb P}}
 
\title{The Spread of a Catalytic Branching Random Walk}

\author{Philippe Carmona\thanks{
 Laboratoire Jean Leray, UMR 6629
 Universit{\'e} de Nantes, BP 92208,
 F-44322 Nantes Cedex 03
 \url{http://www.math.sciences.univ-nantes.fr/\~carmona}\; Supported
 by Grant ANR-2010-BLAN-0108}
\and Yueyun Hu\thanks{Département de Mathématiques (Institut Galilée, L.A.G.A. UMR 7539)
Université Paris 13. \url{http://www.math.univ-paris13.fr/\~yueyun/}\:  supported by ANR  2010 BLAN 0125}
}

\renewcommand{\prob}[1]{\mathbb{P}\etp{#1}}

\newcommand{\hp}{{\mathbb{Q}}}

\newcommand{\qesc}{q_{esc}}
\begin{document}
\maketitle

\begin{abstract}
We consider a catalytic branching random walk on $\Z$ that branches at the origin
only. In the supercritical regime we establish a law of large number
for the maximal position $M_n$: For some   constant $\alpha$, ${M_n
  \over n} \to \alpha$ almost surely on the set of infinite number of  visits of the origin. Then we determine  all possible limiting laws for   $M_n - \alpha n$ as $n $ goes to infinity.  

\end{abstract}

\bigskip


\textbf{Mathematic Classification} : 60K37

\newpage

\section{Introduction}
A    catalytic branching random walk (CBRW) on $\Z$ branching at the
origin only is the following particle system:

  When a particle location $x$ is not the origin,  the particle  evolves
as an irreducible   random walk $(S_n)_{n\in\N}$ on $\Z$ starting from $x$.

  When a particle reaches the origin, say at time $t$, then a time $t+1$
it dies and gives birth to new particles positioned according to a
 point process $\Drond_0$. Each particle (at the origin at time
$t$) produces new particles independently of  every particle living
in the system up to time $t$. These new particles evolve as independent copies of $(S_n)_{n\in \N}$ starting from their birth positions.

The system starts with an initial ancestor particle located at the
origin. Denote by $\PP$ the law of the whole system ($\PP$ also governs  the law of  the underlying random walk $S$), and by $\PP_x$ if the initial particle is located at $x$ (then $\PP=\PP_0$). 
 
Let $\ens{X_u, \valabs{u}=n}$ denote the positions of the particles
alive at time
$n$ (here $\valabs{u}=n$ means that the generation of the particle $u$
in the \emph{Ulam-Harris} tree is $n$). We assume that 
$$ \Drond_0 = \ens{X_u, \valabs{u}=1} \egaloi \ens{S^{(i)}_1, 1 \le i\le N}$$
where $N$ is an integer random variable describing the offspring of a
branching particle, with finite mean $m=\esp{N}$, and $(S^{(i)}_n,n\ge
0)_{i\ge 1}$ are independent copies of $(S_n, n \ge 0)$, and
independent of  $N$.

Let $\tau$ be the
first return time to the origin
$$ \tau := \inf\ens{n\ge 1 : S_n = 0}\quad\text{with $\inf \emptyset =+\infty$}\,.$$

The escape probability is $\qesc := \prob{\tau = +\infty} \in [0, 1)$ ($\qesc < 1$ because $S$ is  irreducible).    Assume that we are in the \emph{supercritical regime}, that is 
\begin{equation}\label{eq:surcritique}
 m (1-\qesc) >1 \,.
\end{equation}
An explanation of assumption~\eqref{eq:surcritique} is given in
Section~\ref{sec:extens-mult-catalyst}, Lemma~\ref{lem:extens-mult-catalyst-1}.

\smallskip
Since the function defined on $(0,\infty)$ by  $r\to
\rho^{(r)} = m \esp{e^{-r \tau}}$
 is of class $C^\infty$, strictly decreasing, $\lim_{r\to
  0}\rho^{(r)}= m \prob{\tau < +\infty} = m (1-\qesc) >1$ and
$\lim_{r\to +\infty} \rho^{(r)}=0$, there exists a unique
$r>0$, a \emph{Malthusian parameter} such that 
\begin{equation}\label{def:malthusian}
  m \esp{e^{-r \tau}} =1\,.
\end{equation}

Let $\psi$ be the logarithmic  moment generating function of $S_1$: $$ \psi(t):= \log \esp{ e^{t S_1}} \in ( - \infty, + \infty], \qquad t \in \R.$$ Let $  \zeta:= \sup\{t >0: \psi(t) < \infty\}.$   We
assume furthermore that  $\zeta>0$ and there exists some $t_0 \in (0, \zeta)$ such that  \begin{equation}\label{hyp1}   \psi(t_0)=r. \end{equation}

\noindent  Observe that by convexity $\psi'(t_0)>0$. 

Let $M_n :=\sup_{\valabs{u}=n} X_u$ be the maximal position at time
$n$ of all living particles (with convention $\sup \emptyset
:=-\infty$).  Since the system only branches at the  origin $0$, we
define the set of infinite number of   visits of the catalyst  by $$ \Srond:=  \ens{\omega :  
    \limsup_{n \to \infty} \ens{u:   \valabs{u}=n, X_u=0}\neq \emptyset}\,.$$

Remark that $\prob {d\omega}$-almost surely on $\Srond^c$,  for all large  $n \ge n_0(\omega)$, either the system dies out or the system behaves  as a finite union of some random walks on $\Z$, starting respectively from $X_u(\omega)$ with $\vert u \vert = n_0$. In particular, the almost sure behavior of $M_n $ is trivial on $\Srond^c$.  It is then natural to consider $M_n$ on the set $\Srond$.  Our first result on $M_n$ is

\begin{theo}[Law of large numbers]\label{thm:lln} Assume \eqref{eq:surcritique} and \eqref{hyp1}. 
  On the set   $\Srond$, 
 we have the convergence
$$ \lim_{n\to +\infty} \frac{M_n}{n} = \alpha:= \frac{\psi(t_0)}{t_0} \quad a.s.$$
\end{theo}

\medskip

In Theorem \ref{thm:lln}, the underlying random walk $S$ can be
periodic.  In order to  refine this convergence to a fluctuation
result by centering $M_n$, we shall need to assume  the  aperiodicity
of $S$. However, we cannot expect a convergence in distribution for
$M_n - \alpha n$ since $M_n $ is integer-valued whereas $\alpha n$ in
general  is not.  

 For $x \in \R$, let  $\lfloor x\rfloor$ be the integer part of $x$ and $\{x \}:= x - \lfloor x \rfloor \in [0, 1)$ be the fractional part of  $x$. 

\begin{theo}\label{thm_deuxsec:introduction}  Assume
  \eqref{eq:surcritique} and \eqref{hyp1}. Assume furthermore that
  $\E(N^2)< \infty$ and  that $S$ is   aperiodic.   Then there exists a constant $c_*>0$ and  a random variable $\Lambda_\infty$
  such that  for any fixed $y \in \R$,  \begin{equation} \label{thm2:a}\prob{ M_n - \alpha n > y} = \esp{1 - e^{-c_*
    e^{-t_0 y}  ( e^{ t_0 \{ \alpha n + y\} } + o(1)) \Lambda_\infty}} , \end{equation}
    where $o(1) $ denotes some deterministic term which goes to $0$ as $n \to \infty$.  The random variable $\Lambda_\infty$ is non negative and satisfies  that \begin{gather}
    \label{eq:cvferchetcond}
    \ens{\Lambda_\infty >0} = \Srond \quad a.s.  
  \end{gather}
 Consequently for any subsequence   $n_j \to \infty$ such that $\{ \alpha n _j\} \to s \in [0, 1)$ for some $s \in [0, 1)$, we have that   \begin{equation} \label{thm2:b} \lim_{ j \to \infty} \prob{ M_{n_j} - \lfloor \alpha n_j \rfloor = y  } =  \E \left(  e^{ - c_* e^{-t_0(y- s)} \Lambda_\infty} -  e^{ - c_* e^{-t_0(y-1-s)} \Lambda_\infty}  \right) \qquad ( \forall y \in \Z.)\end{equation}\end{theo}

Let us make some remarks on Theorem \ref{thm_deuxsec:introduction}:
\begin{rema} \label{R:1}
\begin{enumerate}
\item The random variable $\Lambda_\infty$ is the limit of the positive
fundamental martingale of Section~\ref{sec:fund-mart}. The value of constant $c_*$ is
given in \eqref{defc*} at the beginning of Section~\ref{sec:refining-convergence}.
 \item The hypothesis   $\E(N^2) <\infty$ might be weakened to $\E( N \log (N+1)) < \infty$, just as the classical  $L \log L$-condition (see e.g. Biggins \cite{biggins})  in the branching random walk.  
\item We do need the aperiodicity of the underlying random walk $S$ in
  the proof of Theorem \ref{thm_deuxsec:introduction}.  However, for
  the particular case of  the nearest neighborhood random walk (the
  period equals $2$), we can still get   a modified version of
  Theorem \ref{thm_deuxsec:introduction},  see Remark \ref{R:5} of subsection~\ref{ubp2}.
\end{enumerate} 
  \end{rema}

 \medskip
Theorems 1.1 and 1.2 are new, even though a lot of attention has been
given to CBRW in continuous time. In papers
\cite{MR2144864,MR2090528,MR2042398,MR2110900,MR1740207,MR1650599,MR1649878,MR1702745}
very precise asymptotics are  established for the moments of
$\eta_t(x) $ the number of particles located at $x$ at time $t$, in
every regime (sub/super/critical). Elaborate limit theorems were obtained for the critical case by
Vatutin, Topchii and Yarovaya in
\cite{MR2144864,MR2090528,MR2042398,MR2110900}.

Concerning on  the maximal/minimal position of a  branching random
walk (BRW) on $\R$,  some  important progress were made in recent
years, in particular a convergence in law result was proved in
A{\"{\i}}d{\'e}kon \cite{aidekon1} when the BRW is not lattice-valued.
It is expected that such convergence dos not hold in general for
lattice-valued BRW, for instance see Bramson \cite{bramson} where  he
used a centering with the integer part of some (random) sequence. In
the recent studies of BRW, the spine decomposition technique plays a
very important role. It turns out that a similar spine decomposition
exists for CBRW (and more generally for  branching Markov chains),  and we   especially acknowledge the paper \cite{DoeringRobertsArxiv1106}  that introduced us  the techniques     of multiple spines, see Section~\ref{sec:many-few-formulas}. 

We end this   introduction by comparing our results to
their analogue for (non catalytic) branching random walks (see  e.g.    \cite{aidekon1, hushi, MR1645302,MR2728431}). We shall restrict ourselves to simple
random walk on $\Z$, that is $\prob{S_1=\pm 1} = \undemi$.

For supercritical BRW ($m>1$), almost surely on the set of non
extinction $\lim_{n\to +\infty} \frac{M^{(brw)}_n}{n} = b$, where $b$ is the
unique solution of $\psi^*(b) = \log m$, with $\psi^*(b):= \sup_{t} (b t
-\psi(t))$ the rate function for large deviations of the simple random
walk and $\psi(t) = \log\cosh(t)$. For CBRW, we can do explicit computations : Since for $x\neq 0$, 
$\esperance{x}{e^{-r\tau}} = e^{-t_0 \valabs{x}}$ the Malthusian
parameter satisfies $r+ t_0 = \log(m)$. Combined with $\log \cosh (t_0)
= r$ this implies $e^{t_0}=\sqrt{2m-1}$ and $\alpha = \frac{2
  \log(m)}{\log(2m-1)} -1$.  Numerically, for $m=1.83$ we find $b=0.9$ and $\alpha = 0.24$.
 The second order results emphasize the difference between BRW and CBRW : for BRW, $M^{(brw)}_n- b n$ is of order $O(\log n)$, whereas for CBRW, $M_n- \alpha n$ is of order $O(1)$, see Remark \ref{R:5}. 

The organization of the rest of this  paper is as follows:  We first give in Section~\ref{sec:heuristics} the heuristics explaining the differences  between CBRW and ordinary BRW (branching random walk). Then we proceed
(in Section~\ref{sec:many-few-formulas})
to recall  many to one/few lemmas, we exhibit a fundamental martingale
(in Section~\ref{sec:fund-mart})
and prove Theorems 1.1 and 1.2 in sections 5 and 6 respectively, with the help of sharp asymptotics derived
from renewal theory.  Finally, Section~\ref{sec:extens-mult-catalyst} is devoted to an extension to the case of
multiple catalysts. There the supercritical assumption
\eqref{eq:surcritique} appears in a very natural way.

Finally,  let us denote by $C$, $C'$ or $C^{''}$ some unimportant positive constants whose values can be changed from one paragraph to an another.


\section{Heuristics}\label{sec:heuristics}
Assume for sake of simplicity that we have  a simple random walk. The  existence
of the fundamental martingale $\Lambda_n = e^{-rn} \sum_{\valabs{u}=n}
\phi(X_u)$, See Section~\ref{sec:fund-mart}, such that $\ens{\Lambda_\infty
  >0}=\Srond$,  shows that on the set of non extinction $\Srond$, we have
roughly $e^{rn}$ particles at time $n$.

If we apply the usual heuristic for branching random walk (see
e.g. \cite{MR1645302} Section II.1), then we say
that we have approximately $ e^{rn}$ independent random walks
positioned at time $n$, and therefore the expected population above
level $an>0$ is roughly:
\[ \esp{\sum_{i=1}^{\lfloor e^{rn}\rfloor} \un{S^{(i)}_n \ge an}} = \floor e^{rn}\rfloor 
\prob{S_n \ge an} = e^{-n ( \psi^*(a) -r)(1+o(1))} \]
where $ \psi^*(a) = \sup_{t\ge 0} (ta -\psi(t))$ is the large
deviation rate function (for simple random walk, $e^{\psi(t)}=\esp{e^{t
    S_1}} = \ch(t)$).

This expected population is of order $1$ when $\psi^*(a)=r$ and therefore
we would expect to have $\frac{M_n}{n} \to \gamma$ on $\Srond$, where
$\psi^*(\gamma)=r$.

\medskip
However, for CBRW, this is not the right speed, since the positions of
the independent particles cannot be assumed to be distributed as
random walks. Instead, the $\lfloor e^{rn}\rfloor$ independent particles may be
assumed to be distributed as a fixed probability distribution, say 
$\nu$. If $\eta_n(x)=\sum_{\valabs{u}=n} \un{X_u=x}$ is the number of
particles at location $x$ at time $n$, we may assume that  for a
constant $C>0$, $e^{-rn}
\esp{\eta_n(x)} \to C \nu(x)$ and thus, $\nu$ inherits from $\eta_n$ the
relation :
$$ \nu(x) = e^{-r} \sum_y c(y) p(y,x) (m \un{y= 0} + \un{y\neq 0})$$
with $p(x,y)$ the random walk kernel. For simple random walk, this
implies that for $\valabs{x}\ge 2$ we have  $\undemi(\nu(x+1)
+\nu(x-1)) = e^{r} \nu(x)$ and thus $\nu(x)
= C e^{-t_0\valabs{x}}$ for $\valabs{x} \ge 2$, with $\psi(t_0)= \log\cosh(t_0)=r$.

 Therefore
the expected population with distance to the origin at least $an$ is
roughly
\begin{align*}
  \esp{\sum_{\valabs{x} \ge an} \eta_n(x)} &= e^{rn} \sum_{\valabs{x}
    \ge an} e^{-rn} \esp{\eta_n(x)} \\
&\sim e^{rn} C \sum_{\valabs{x}
    \ge an} e^{-t_0 \valabs{x}} \sim C' e^{rn} e^{-t_0 a n}
\end{align*}

This expectation is of order $1$ when $a= \frac{r}{t_0} =
  \frac{\psi(t_0)}{t_0} = \alpha$, and this yields the right
  asymptotics
$$ \frac{M_n}{n} \to \alpha \quad{\text{a.s. on $\Srond$}}\,.$$
This heuristically gives the law of large numbers in Theorem
\ref{thm:lln}. 



\section{Many to one/few formulas for multiple catalysts branching random walks (MCBRW)}\label{sec:many-few-formulas}

For a detailed exposition of many to one/few formulas and the spine
construction we suggest the papers of  Biggins and Kyprianou \cite{bk04}, Hardy and
Harris \cite{MR2599214}, Harris and Roberts
\cite{HarrisRobertsArxvi2011}  and the references
therein. For an application to the computations of moments asymptotics
in the continuous setting, we refer to D\"{o}ring and
Roberts \cite{DoeringRobertsArxiv1106}.  We state the many to one/two formulas for a CBRW with multiple catalysts and will specify the formulas in the case with a single catalyst. 


\subsection{Multiple catalysts branching random walks (MCBRW)}\label{subsec1}

The set of catalysts is  a some  subset $\Crond$ of $\Z$.  When a
particle reaches a catalyst $x\in\Crond$ it dies and gives birth to
new particles according to the point process
$$ \Drond_x \egaloi (S^{(i)}_1,1\le i\le N_x)$$
where $(S^{(i)}_n,n\in\N)_{i\ge 1}$ are independent copies of an irreducible random walk $(S_n, n \in \N)$  starting form
$x$, independent of  the random variable $N_x$ which is  assumed to be   
integrable. Each particle in $\Crond$  produces new particles independently from the other
particles living in the system. Outside of
$\Crond$ a particle performs a    random walk distributed as $S$. The CBRW (branching only at $0$) corresponds to $\Crond= \{ 0\}$.

\subsection{The many to one formula for MCBRW}





 Some of the most interesting results about first and second moments
 of particle occupation numbers   that we obtained come from the
 existence of a ``natural'' martingale. An easy way to transfer martingales from the random walk   to the branching processes   is to use a slightly extended many to one formula that enables conditioning.  Let \begin{equation}\label{defm1x}  m_1(x): =\esp{N_x} < \infty ,   \qquad x \in \Z. \end{equation}
 On the space of trees with a spine (a distinguished line of
descent) one can define a probability  $\Q$ via martingale change of
probability, that satisfies
\begin{equation}\label{mtonecondgbrw}
 \esp{ Z \sum_{\valabs{u}=n} f(X_u)} = \Q\etc{ Z f(X_{\xi_n})
  \prod_{0\le k\le n-1} m_1(X_{\xi_k})}\,,
\end{equation}
for all $n\ge1$,  $f: \Z \to \R_+$   a nonnegative function and $Z$   a positive $\Frond_n$ measurable random variable, and  where  $(\Frond_n, n\ge 0)$ denotes the natural filtration generated by the MCBRW  (it does not contain information
about the spine).   On the right-hand-side of \eqref{mtonecondgbrw} $(\xi_k) $ is the spine, and
it happens that the distribution of $(X_{\xi_n})_{n\in\N}$ under $\Q$
is the distribution of the random walk $(S_n)_{n\in\N}$.

Specializing this formula to CBRW for which $m_1(x) = m\un{x=0} +
\un{x\neq 0}$ yields

\begin{equation}\label{mtonecbrw}
\esp{\sum_{\valabs{u}=n} f(X_u)} = \esp{f(S_n)m^{L_{n-1}}}\,,
\end{equation}
where $L_{n-1}=\sum_{k=0}^{n-1} \un{S_k=0}$ is the local time at level
$0$.

\subsection{The many to two formula for MCBRW}

Recall \eqref{defm1x}.  Let us assume that \begin{equation}\label{defm2x}    m_2(x):=\esp{N_x^2} < \infty , \qquad x \in \Z. \end{equation}
Then for any  $n \ge 1$ and $f : \Z\times \Z \to \R_+$, we have

\begin{equation}
  \label{mtotogbrw}
  \esp{\sum_{\valabs{u}=\valabs{v}=n} f(X_u,X_v)} =
  \Q\etc{f(S^1_n,S^2_n) \prod_{0\le k< T^{de}\wedge n} m_2(S^1_k)
    \prod_{T^{de}\wedge n \le k < n}m_1(S^1_k) m_1(S^2_k)}\,,
\end{equation}
where under $\Q$, $S^1$ and $S^2$ are coupled random walks that start
from $0$ and stay coupled (in particular at the same location)  until  the decoupling time $T^{de}$ and after $T^{de}$, they   behave as
independent random walks. 

More precisely, we have a three component Markov process
$(S^1_n,S^2_n,I_n, n\ge 0)$ where $I_n \in \ens{0,1}$ is the indicator
that is one iff the random walks are decoupled: when the two random walks are
coupled at time $n$, and  at site $x$, the they stay coupled at time
$n+1$ with probability $\frac{m_1(x)}{m_2(x)}$. That means that the
transition probability are the following:
\begin{itemize}
\item
$\prob{S^1_{n+1}=y, S^2_{n+1}=y, I_{n+1}=0 \mid S^1_n=S^2_n=x, I_n=0}
= \frac{m_1(x)}{m_2(x)} p(x,y)$
\item $\prob{S^1_{n+1}=y, S^2_{n+1}=z, I_{n+1}=1 \mid S^1_n=S^2_n=x, I_n=0}
= (1-\frac{m_1(x)}{m_2(x)}) p(x,y) p(x,z)$
\item  $\prob{S^1_{n+1}=y, S^2_{n+1}=z, I_{n+1}=1 \mid S^1_n=x_1,S^2_n=x_2, I_n=1}=p(x_1,y)p(x_2,z)$.
\end{itemize}

The random walks are initially coupled and at the origin. The decoupling time $T^{de}=\inf\ens{n\ge 1 : I_n =1}$ satisfies  for any $k \ge0$, 
\begin{equation}\label{decouple} \Q\etc{T^{de}\ge k+1 \mid \sigma\{ S^1_j, S^2_j, I_j,  j\le k\}} = \prod_{0\le l\le k-1}
\frac{m_{1}(S^1_l)}{m_2(S^2_l)} \, \un{I_k=0}\,,\end{equation} 
where we keep the usual convention $\prod_\emptyset\equiv 1$.

This formula is proved in \cite{MR2599214,HarrisRobertsArxvi2011} by
defining a new probability $\Q$ on the space of trees with two spines.

An alternative proof, that makes more natural the coupling of
$(S^1,S^2)$ is to condition on the generation of the common ancestor
$w=u\wedge v$ of the two nodes, then use the branching to get
independence, and plug in the many to one formula in each factor. We omit the details. 

\section{A fundamental Martingale}\label{sec:fund-mart}
Martingale arguments have been used for a long time in the study of
branching processes. For example, for the Galton Watson process with
mean progeny $m$ and population $Z_n$ at time $n$, the sequence $W_n
=\frac{Z_n}{m^n}$ is a positive martingale converging to positive
finite random variable $W$. The Kesten-Stigum theorem implies that if
$\esp{N\log (N+1)} < +\infty$, we have the identity a.s., $\ens{W>0} $ equals the survival set. A classical proof can be found
in the reference book of Athreya and Ney \cite{MR2047480}, Section
I.10. A more elaborate proof, involving size-biased branching
processes, may be found in Lyons-Pemantle-Peres \cite{MR1349164}.

Similarly, the law of large numbers for the maximal position $M_n$ of
branching random walks system may be proved by analyzing a whole one
parameter family of martingales (see  Shi \cite{refId0} for a detailed exposition on the equivalent form of Kesten-Stigum's  theorem for BRW).
Recently, the maximal position of a branching brownian motion with
inhomogeneous spatial branching has also been studied with the help a
a family of martingale indexed this time by a function space (see
Berestycki, Brunet, Harris and Harris \cite{MR2669786} or Harris and
Harris \cite{MR2548504}).

\medskip
We want to stress out the fact that for catalytic branching random
walk, since we branch at the origin only, we only have one natural
martingale, which we call the fundamental martingale.

Let $T=\inf\ens{n\ge 0 : S_n =0} $ be the first hitting time of $0$,
recall that $\tau = \inf\ens{n\ge 1 : S_n=0}$ and let
\begin{equation} \label{defphix} \phi(x) := \esperance{x}{e^{-r T}} \quad(x\in\Z)\,, \end{equation}
where $r$ is given in \eqref{def:malthusian}. 
Finally let $p(x,y)=\PP_x\etp{S_1=y}$ and $Pf(x)=\sum_y p(x,y) f(y)$
be the kernel and semigroup of the random walk $S$.

\begin{prop} \label{P:41} Under \eqref{eq:surcritique} and \eqref{hyp1}. 
  \begin{enumerate}
  \item The function $\phi$ satisfies
$$ P\phi(x) = e^r \phi(x) \etp{\unsur{m} \un{x=0} + \un{x\neq 0}}\,.$$
\item The process 
$$ \Delta_n := e^{-rn} \phi(S_n) m^{L_{n-1}}$$ is a martingale, where
$L_{n-1}= \sum_{0\le k\le n-1} \un{S_k=0}$ is the local time at level $0$.
\item The process
$$ \Lambda_n := e^{-rn} \sum_{\valabs{u}=n} \phi(X_u)$$
is a martingale called the \underline{fundamental martingale}.
\item If $\esp{N^2} < +\infty$, then the process $\Lambda_n$ is bounded in $L^2$, and therefore is a
  uniformly integrable martingale.
  \end{enumerate}
\end{prop}
\begin{proof}
  (1) If $x\neq 0$, then $T\ge 1$, therefore, by conditioning on the
  first step:
  \begin{align*}
    \phi(x) = \sum_y p(x,y) e^{-r} \esperance{y}{e^{-rT}} = e^{-r} P\phi(x)\,.
  \end{align*}
On the other hand, $\tau \ge 1$ so conditioning by the first step again,
$$ \phi(0)=1 = m\esp{e^{-r\tau}} = m \sum_y p(0,y) e^{-r}
\esperance{y}{e^{-rT}} = m e^{-r} P\phi(0)\,.$$

(2)  Denote by $\Frond^S_n:=\sigma\{ S_1, ..., S_n\}$ for $n\ge1.$ We have,
\begin{align*}
\esp{\Delta_{n+1} \mid \Frond^S_n} &= e^{-r(n+1)} m^{L_n}
\esp{\phi(S_{n+1}) \mid \Frond^S_n} = e^{-r(n+1)} m^{L_n} P\phi(S_n) \\
&= e^{-r(n+1)} m^{L_n} e^{r} \phi(S_n)(\unsur{m} \un{S_n=0} +
\un{S_n\neq 0}) = \Delta_n\,.
\end{align*}
(3) Recall that $(\Frond_n)_{n\ge0}$ denotes the natural filtration of the CBRW.     By the many to one formula, if $Z$ is $\Frond_{n-1}$ measurable positive,  then   
\begin{align*}
  \esp{\Lambda_n Z} &= e^{-rn} \esp{ \sum_{\valabs{u}=n} \phi(X_u) Z}
  \\
&=e^{-rn} \esp{Z \phi(S_n) m^{L_{n-1}}} = \esp{Z \Delta_n} \\
&= \esp{Z \Delta_{n-1}}  \quad\mbox{(the martingale property of $\Delta_n$)}\\
&=\esp{\Lambda_{n-1} Z}\,.
\end{align*}
(4) The proof is given in Section~\ref{sec:extens-mult-catalyst} in
the case of multiple catalysts and uses heavily the many to two
formula. 
\end{proof}

Let us introduce $\eta_n(x)$ the number of particles located at $x$ at
time $n$:
$$ \eta_n(x) := \sum_{\valabs{u}=n} \un{X_u=x}\,.$$
\begin{coro}\label{cor:secmomentetan} Under \eqref{eq:surcritique} and \eqref{hyp1}.
  \begin{enumerate}
  \item We have $\sup_{x,n} e^{-rn} \phi(x)\eta_n(x) < +\infty$ a.s.
  \item If $N$ has finite variance then there exists a constant $0<C<\infty$ such that
$$ \esp{\eta_n(x)\eta_m(y)} \le \frac{C}{\phi(x)\phi(y)}\; e^{r(n+m)}
\quad (n,m\in\N, x,y\in\Z^d).$$
  \end{enumerate}
\end{coro}
\begin{proof}
  (1) Let us write $\Lambda_n = e^{-rn} \sum_x \phi(x)
  \eta_n(x)$. Since it is a  positive martingale it converges almost
  surely to a finite integrable positive random variable
  $\Lambda_\infty$. Therefore $\Lambda^*_\infty := \sup \Lambda_n <
  +\infty$ a.s.and
$$ \sup_{x,n} e^{-rn}\phi(x)\eta_n(x) \le \Lambda^*_\infty\,.$$

(2) Assume for example that $n\le m$ and let
$C=\sup_n\esp{\Lambda_n^2} < +\infty$. We have, since $\Lambda_n$ is a martingale,
\begin{align*}
  e^{-r(n+m)}\phi(x) \phi(y) \esp{\eta_n(x)\eta_m(y)} &\le \esp{\Lambda_n
    \Lambda_m} \\
&= \esp{\Lambda_n\esp{\Lambda_m \mid \Frond_n}}=\esp{\Lambda_n^2} \le C.
\end{align*}

\end{proof}

For the proof of the following result instead of using large
deviations for $L_n$, we  use renewal theory, in the spirit of
\cite{MR2658973,ldpcar}. Let $d$ be the period of the return times to
$0$: \begin{equation} \label{defperiod} d := \gcd\ens{n\ge 1 : \prob{\tau =n} >0} \,. \end{equation}

\begin{prop}\label{prop:renewal:etanx} Assume \eqref{eq:surcritique} and \eqref{hyp1}.
  For every     $x\in\Z$ there exists a constant $c_x\in (0,\infty)$ and a unique $  l_x \in\ens{0,1, \cdot\cdot\cdot, d-1}$ such
  that
$$ \lim_{n\to +\infty} e^{-r(dn+l_x)} \esp{\eta_{nd+l_x}(x)} = c_x\,.$$
Moreover, for any $l \not \equiv l_x  \mbox{ (mod $d$)} $, $\eta_{nd+l}(x)=0$ for all $n\ge0$.  In particular, for $x=0$,  $l_x=0$ and $c_0= {d \over m}$. 
\end{prop}

\begin{proof}
  By the many to one formula \eqref{mtonecondgbrw}, 
  \begin{align*}
    v_n(x) &:= \esp{\eta_n(x)} = \esp{\sum_{\valabs{u}=n} \un{X_u=x}}
    \\
&=\hp\etp{\un{S_n=x} e^{A_0(\xi_n)}} \\
&= \esp{\un{S_n=x} m^{L_{n-1}}}\,.
  \end{align*}
We decompose this expectation with respect to the value of $\tau
=\inf\ens{n\ge 1: S_n=0}$:
$$ v_n(x) = m\esp{\un{S_n=x} \un{\tau \ge n}} + \sum_{1\le k\le
  n-1}\esp{\un{S_n=x} m^{L_{n-1}}\un{\tau=k}}.$$
By the Markov property, if $u_k := \prob{\tau=k}$, then
\begin{align*}
  v_n(x) = m \prob{\tau\ge n, S_n=x} + \sum_{1\le k\le n-1} m u_k
  v_{n-k}(x) 
= m \prob{\tau\ge n, S_n=x} + m v.(x)*u(n)\,,
\end{align*}
Recall that the Malthusian
  parameter $r$ is defined by
$$ 1 = m \esp{e^{-r\tau}} = m \sum_{k\ge 1} e^{-rk} u_k\,.$$
Hence if we let $\tilde{v}_n(x) = e^{-rn} v_n(x) $ and $\tilde{u}_k =
m e^{-rk} u_k$ then,
$$ \tilde{v}_n(x) = m e^{-rn} \prob{\tau\ge n, S_n=x} +
\tilde{v_\cdot}(x)*\tilde{u}(n)\,.$$

By the periodicity, we have $u_{n}=0$ if $n$ is not a multiple of
$d$ and for $x \in \Z^d$ there is a unique  $l_x\in \ens{0, 1, \ldots, d-1}$ such that $\nu_n(x) =0$ if $n \not \equiv l_x \mbox{ (mod $d$)}$. Therefore the
sequence $t_n=\tilde{v}_{nd+l}(x)$ satisfies the following renewal
equation

$$ t_n = y_n + t*s_n$$
with $s_n = \tilde{u}_{nd}$ and $y_n= e^{-r(nd+l_x)} \prob{\tau \ge dn+l_x,
  S_{dn+l_x}=x}$. Since the sequence $s$ is aperiodic, the discrete
renewal theorem (see Feller\cite{MR0038583}, section  XIII.10, Theorem 1) implies that
$$ t_n \to \frac{\sum_{n=1}^\infty y_n}{\sum_{n=1}^\infty n s_n}=: c_x\,.$$
Remark that $\sum_{n=1}^\infty n s_n= \sum_{n=1}^\infty n e^{- r nd } m u_{ nd }={1\over d}$. We have $$ c_x= d \, \sum_{n=1}^\infty e^{-r(nd+l_x)} \prob{\tau \ge dn+l_x,   S_{dn+l_x}=x}>0.$$ This is exactly the desired result.   

Finally for $x=0$, $\ell_x=0$ and $ c_0= d \, \sum_{n=1}^\infty e^{-r nd } \prob{\tau \ge dn ,   S_{dn }=0} = d \, \sum_{n=1}^\infty e^{-r nd } \prob{\tau = dn  }  = d\, \E( e^{- r \tau})= {d \over m}$ by the choice of $r$. This completes the proof of Proposition \ref{prop:renewal:etanx}. 
\end{proof}


\begin{rema}  The family $(c_x)_{x \in \Z}$ satisfies a system of linear equations, dual to the one   (see Proposition \ref{P:41})
   satisfied by the function $\phi$:  Recalling that $p(x,y)=\p_x(S_1=y)$ is the kernel of
  the random walk,   we have the recurrence relation $$\esp{\eta_{n+1}(x)} = \sum_y \esp{\eta_n(y)} p(y,x) (m \un{y=0} +
\un{y\neq 0}).$$
Assuming for simplicity $d=1$ and multiplying by $e^{-r(n+1)}$ and letting $n\to +\infty$, we obtain the
following functional equation for the function $x\to c_x$:

$$ c_x = e^{-r} \sum_y c_y p(y,x) (m \un{y=0} +
\un{y\neq 0})\,,\qquad x \in \Z.$$ \end{rema}

\medskip
We end this section by   the following lemma which  yields   the   part  \eqref{eq:cvferchetcond} in 
Theorem~\ref{thm_deuxsec:introduction}.
\begin{lemm}\label{L:n2} Assume \eqref{eq:surcritique} and \eqref{hyp1}.
  Assume furthermore that  $N$ has finite variance. Then we  have 
$$\ens{\Lambda_\infty >0}= \Srond\quad\text{ a.s.}$$
\end{lemm}

Remark that  in this Lemma we do not need  the aperiodicity of the underlying random walk $S$.

\begin{proof} We first prove that $\Srond^c \subset \ens{\Lambda_\infty=0}$ a.s.
 In fact, on $\Srond^c$, either the system dies out then $\Lambda_n=0$
 for all large $n$, or for all large $n \ge n_0(\omega)$ :
 $\eta_n(0)=0$ . Then, if $\eta_n = \sum_x
 \eta_n(x)$ is the total population, $\eta_n=\eta_{n_0}$  for all $n \ge
 n_0$ since the system only branches at $0$. Since $\Lambda_n = e^{-r n} \sum \phi(X_u) \le e^{-r n} \eta_n = e^{-rn } \eta_{n_0}$, we still get $\Lambda_\infty=0$. 
 
Let $s=\prob{\Lambda_\infty =0}$ and $\hat s:= \prob{ \Srond^c}$. If we can prove $s= \hat s$, then the  Lemma follows. We shall condition on the number of   children of the initial ancestor $N$. For $k\ge j \ge0$,  let $ \Upsilon_{k, j}$ be the event  that amongst $k$ particles of the first generation  there are   exactly $j$ particles  which will return    to $0$. Then 
    \begin{align*}
    s  = \prob{\Lambda_\infty =0} &= \sum_{k=0}^\infty  \prob{N=k} \sum_{j=0}^k \prob{\Upsilon_{k, j}\cap\{ \Lambda_\infty=0\} \big\vert N=k}   \\
    &= \sum_{k=0}^\infty \prob{N=k} \sum_{j=0}^k \binom{k}{j} \qesc^{k-j} (1-  \qesc)^j s^j  \\
&= \sum_{k=0}^\infty  \prob{N=k} ( \qesc + s(1-\qesc )^k =f(q_{esc}+ s(1-\qesc)), 
  \end{align*}
with $f(x)=\esp{x^N}$ the generating function of the reproduction law. Exactly in the same way, we  show   that $\hat s$ satisfies the same equation as $s$. 

It remains to check the equation $x= f(q_{esc}+ x(1-\qesc))$ has a unique solution in $[0, 1)$ ($\hat s \le s$ and $s <1$ thanks to Proposition \ref{P:41}). 
To this end, we consider the function $g(x):= f(q_{esc}+ x(1-\qesc)) -x$. The function  $g$ is strictly convex on $[0,1]$, $g(0)= f(\qesc)>0$, $g(1)=0$ and $g'(1) = m  (1-\qesc)-1>0$. Thus $g$ has a unique zero on $[0,1)$, proving the Lemma. 
\end{proof}

\section{The law of large numbers : proof of Theorem~\ref{thm:lln}.}\label{sec:law-large-numbers}

\xcom{We assume that for all real $t$ : $e^{\psi(t)} = \esp{e^{t S_1}} <
+\infty$.  We assume that $\prob{S_1 >0}>0$ since otherwise the random
walk never goes to the right and we have trivially $M_n=0$ for all
$n$.

We have then $e^{\psi(t)} \ge e^{t} \prob{ S_1>0} \to +\infty$ and
therefore consider the unique $t_0>0$ such that
$$ \psi(t_0) = r\,,$$
with $r$ the Malthusian parameter given by~\eqref{def:malthusian}.

We are going to prove that on the survival set $\Srond$, almost
surely,
$$ \lim_{n\to+\infty} \frac{M_n}{n} = \alpha  := \frac{r}{t_0}\,.$$
}
\subsection{Proof of the upper bound.}
\medskip Let $\theta >0$, $x>0$. By the many to one formula,
\begin{align*}
  \prob{M_n > xn} &= \prob{ \sum_{\valabs{u}=n} \un{X_u > xn}  \not= 0} \\
  &\le \esp{\sum_{\valabs{u}=n} \un{X_u > xn}} \\
  &= \esp{\un{S_n > nx} m^{L_{n-1}}}\\
  &\le \esp{e^{\theta(S_n -xn)} m^{L_{n-1}}} = e^{-\theta n x}  h_n
  &\text{, with } h_n= \esp{e^{\theta S_n} m^{L_{n-1}}}\,.
\end{align*}

As in Proposition~\ref{prop:renewal:etanx}, we are going to use the 
renewal theory to study the asymptotics of $v_n$. Let us condition on
$\tau=\inf\ens{n\ge 1 : S_n=0}$:
\begin{align*}
  h_n &= \esp{e^{\theta S_n} m^{L_{n-1}} \un{\tau \ge n}} +
  \sum_{1\le k\le n-1} \esp{e^{\theta S_n} m^{L_{n-1}}\un{\tau =k}}\\
  &= \esp{e^{\theta S_n} \un{\tau \ge n}} +
  \sum_{1\le k\le n-1} m \prob{\tau =k} h_{n-k} \\
  &= z_n + m h*u(n)\,, 
\end{align*}
with $ z_n :=  \esp{e^{\theta S_n} \un{\tau \ge n}}$ and $u_n := \prob{\tau
  =n}$.

Assume now that $\theta > t_0$ so that $\psi(\theta) > \psi(t_0) =
r$.  We let
$$\tilde{h}_n := e^{-n \psi(\theta)} h_n\,,\quad \tilde{z}_n := e^{-n
  \psi(\theta)} z_n\,,\quad \tilde{u}_n := m e^{-n \psi(\theta)} u_n\,.$$
On the one hand, by definition of the Malthusian parameter we have $1= m \esp{e^{-r
    \tau}} = \sum m_n e^{-rn} u_n$ so that $\sum_k \tilde{u}_k < 1$.

On the other hand, 
$$ \tilde{z}_n = \esp{e^{\theta S_n -n \psi(\theta)} \un{\tau \ge n}}
= \PP_\theta\etp{\tau \ge n}$$
with $\PP_\theta$ defined by the martingale change of probability
$$ \frac{d\PP_\theta}{d\PP}= e^{\theta S_n -n \psi(\theta)}
\quad\text{(on $\Frond_n$)}\,.$$
Since under $\PP_\theta$, $(S_n)_{n\ge 0}$ is a random walk with mean
$\esperance{\theta}{S_1}=\psi'(\theta) \ge \psi'(t_0) >0$, we have
$$ \tilde{z}_n \to \tilde{z}_\infty :=\PP_\theta\etp{\tau = +\infty}\,.$$

If we make  the aperiodicity assumption $d=1$,
then by the discrete 
renewal theorem, we have
$$ \tilde{h}_n \to \frac{\tilde{y}_\infty}{1- \sum_k \tilde{u}_k}\,.$$
In the general case, we can  prove exactly as in the proof of
Proposition~\ref{prop:renewal:etanx} that for every $l\in \ens{0,
  \ldots, d-1}$ there exists a finite constant $K_l$ such that 
$$ \lim_{n\to +\infty} \tilde{h}_{nd+l} \to K_l\,.$$
Therefore in any case, the sequence $\tilde{h}_n$ is bounded, and if $x> \frac{\psi(\theta)}{\theta}$
$$\prob{M_n > x n} \le e^{-n (\theta x -\psi(\theta))} \tilde{h}_n$$
satisfies $\sum_n \prob{M_n > xn} < +\infty$. Hence,  by Borel Cantelli's lemma 
$$\limsup_{n\to +\infty} \frac{M_n}{n} \le x \quad a.s.$$
Hence, letting first $x\downarrow \frac{\psi(\theta)}{\theta}$ and
then $\theta \downarrow t_0$ we obtain that 
$$\limsup_{n\to +\infty} \frac{M_n}{n} \le \frac{\psi(t_0)}{t_0} =
\alpha \quad a.s.$$

\subsection{Proof of the lower bound, under the hypothesis $\e(N^2)<\infty$.} \label{s:lbn2} 

The strategy of proof is as follows: Let $0<s<1$, $a>0$ and consider the
event $\Arond_{n,a,s}$ (with  $c'$ a positive constant): ``the particles survive forever, there are at
least $\undemi c' e^{rsn}$ particle alive at time $sn$, and one of these
particle stays strictly positive until time $n$ and reaches a position
larger that $(1-s) a n$ at time $n$''.

We shall prove that for a suitable constant $c'$, we can choose $a,s$
such that on the   set $\Srond$ of infinite number of  visits to $0$, for large $n$ we are in
$\Arond_{n,a,s}$. This implies that almost surely on $\Srond$, $\liminf
\frac{M_n}{n}\ge a(1-s)$. Optimizing over the set of admissible
couples $(a,s)$ will yield the desired lower bound : $\liminf
\frac{M_n}{n}\ge \alpha$ a.s. on $\Srond$.

\bigskip
Recall from Proposition~\ref{prop:renewal:etanx} and
Corollary~\ref{cor:secmomentetan}  that
$$ \lim_{n\to +\infty} e^{-r d n} \esp{\eta_{dn}(0)} = c_0\,,\quad \sup_n e^{-2r d n}
\esp{\eta_{dn}(0)^2} <+\infty\,.$$ Therefore Paley-Zygmund's inequality entails
that \begin{equation}\label{pz1}  \p\Big( \eta_{dn}(0) \ge c'\, e^{r d n} \Big) \ge c', \end{equation} for some
constant $c' >0$.  The following lemma
aims at describing the a.s. behavior of $\eta_n(0)$:

\begin{lemm} \label{lem:proof-lower-bound} Under \eqref{eq:surcritique} and \eqref{hyp1}. Almost surely on ${\cal S}$, $$ \eta_{dn}(0) \ge {c'\over 2}
  e^{r d n},$$ for all large $n$.
\end{lemm}

\begin{proof}

We shall write the proof for the aperiodic case $d=1$. The
 generalization to a period $d\ge 2$ is  straightforward by considering $d n$ instead of $n$ throughout the proof of this Lemma.

Let $\eta_n =  \sum_x \eta_n(x) $ be the total population at time
$n$. Since $0\le \phi(x) \le 1$ we have $\Lambda_n = e^{-rn} \sum_x
\phi(x) \eta_n(x) \le e^{-rn} \eta_n$. Furthermore, a particle living at time $n$ has to have an ancestor at
location  $0$ at some
time $k\le n$, and  if $N_i$ is the number of children of this
ancestor, then 
$$ \eta_n \le \sum_{1\le i\le \Gamma_n} N_i \quad \text{with}\quad
\Gamma_n=\sum_{0\le k\le n} \eta_k(0)$$
where the $(N_i)_{i\ge 1} $ are independent random variables
distributed as $N$ and independent of $\Gamma_n$. Since $\esp{N}<
+\infty$, by Borel Cantelli's Lemma, there exists  $i_0=i_0(\omega)$ such that

$$ N_i \le   i^2 \quad\text{for}\quad i\ge i_0\,.$$
Hence, almost surely for $n$ large enough,
\begin{align*}
  \eta_n &\le \sum_{1\le i\le i_0} N_i + \Gamma_n^2 \\
&\le  \sum_{1\le i\le i_0} N_i +   n^2 (\sup_{0\le k\le n}
\eta_k(0))^2
\end{align*}

By Lemma \ref{L:n2}, almost surely on the survival set $\Srond$, we have
$\Lambda_\infty >0$ and thus, for $n$ large enough $\eta_n \ge \undemi
\Lambda_\infty e^{r n} $ and therefore for $n$ large enough, on $\Srond$, 
$$ \sup_{0\le k\le n}
\eta_k(0)   > e^{ rn/4}$$
 Considering the stopping time (for the
branching system endowed with the natural filtration) $$ T:=
\inf\{n : \eta_n(0) > e^{ rn/4}\}.$$

\noindent We have established that on $\Srond$, $T< \infty $ a.s.  It
follows from the branching property and \eqref{pz1} that \begin{eqnarray*} \p \Big(
  \eta_{n+T}(0) \le c' \, e^{rn} , \, {\cal S}\Big) &\le& \p\Big(
  \eta_n(0) \le c' \, e^{rn} \Big) ^{ e^{ rn/4}} \\
  &\le& (1-c')^{e^{ r n/4}}, \end{eqnarray*}

\noindent whose sum on $n$ converges. By Borel-Cantelli's lemma, on
${\cal S}$, a.s. for all large $n$, $$ \eta_n(0) \ge c' e^{r(n-
  T)} \ge { c'\over 2} e^{rn}.$$
This proves the Lemma.
\end{proof}

\begin{proof}[Proof of the lower bound of $M_n$] Let $0< s <1$.  Define $k=k(n):= d \lfloor {s n\over d} \rfloor$.
  By the preceding Lemma, on
the survival set ${\cal S}$, at time $k $, there are at least $ \lfloor {c'\over 2} \,
e^{r k }\rfloor $ particles at $0$, which move independently. Letting these
particles move as  the random walk  $S$ staying positive up to time
$n-k $, then $M_n$ is bigger than $ \lfloor {c'\over 2} \,
e^{r k }\rfloor $ i.i.d.
copies of $S_{n-k }$ with $S_1>0, ..., S_{n-k} >0$. By a large
deviations estimate (Theorem 5.2.1 of Dembo
and Zeitouni \cite{MR2571413}, see the forthcoming Remark \ref{mogulskii}),  for any fixed $a \in (0, \infty)$, $$ \p \Big( S_{ n-k } > a
(1-s) n, S_1>0, ..., S_{n-k} >0 \Big) = e^{- (1-s) n \psi^*(a)
  +o(n)}, $$ where we denote as before,  $$ \psi^*(a) = \sup_{\theta >0} ( a \theta -
\psi(\theta))\,.$$

It follows that
\begin{eqnarray*}&& \p \Big( M_n \le  (1-s) a n, \eta_{k }(0) \ge {c'\over 2} \, e^{r k} \Big) \\ &\le & \Big( 1- \p ( S_{ n-k } > a (1-s) n, S_1>0, ..., S_{n-k} >0 \Big)\Big)^{ \lfloor {c'\over 2} \, e^{r k } \rfloor}  \\
  &= & \exp(- e^{rsn  - \psi^*(a) (1-s) n +o(n)})
  .\end{eqnarray*}

Choose $(a, s) \in (0,+\infty)\times(0,1)$ such that $$r s > \psi^*(a)
(1-s), $$ we apply Borel-Cantelli's lemma and get that a.s. for all
large $n$, either $M_n > (1-s) a n$ or $\eta_{k}(0) < {c'\over 2} \,
e^{rk} $. Hence on the   set ${\cal S}$, by Lemma~\ref{lem:proof-lower-bound}, a.s., 
\begin{equation}\label{lb111}
\liminf_{n\to \infty} { M_n \over n } \ge \gamma:=\sup\{ (1-s )a: (a, s) \in
(0,\infty)\times (0,1), r s >  \psi^*(a) (1-s)\} .  \end{equation}

Recalling $r= \psi(t_0)$, then 
$$ \gamma = \sup_{a>0} { a \psi(t_0) \over \psi^*(a) + \psi(t_0)}.$$

Let us study the derivative of $a \to { a \psi(t_0) \over \psi^*(a) +
  \psi(t_0)}$. Recall that  $ \psi^*(a) = a \theta(a) - \psi( \theta(a))$ with $a=
\psi'(\theta(a))$, and  $(\psi^*)'(a)= \theta(a)$. Since the derivative of
$a \to { a \psi(t_0) \over I(a) + \psi(t_0)}$ has the same sign as $
\psi^*(a) + \psi(t_0) - a (\psi^*)'(a)= \psi(t_0) - \psi(\theta(a)) $, it is
negative if $ a > \psi'( t_0) $ (i.e. $\theta(a) > t_0$), positive if
$a< \psi'( t_0) $ and vanishes  at $\psi'( t_0) $. Therefore
 $$ \gamma = { \psi(t_0) \over t_0}=\alpha,$$ which in view of \eqref{lb111} yields  the lower bound of Theorem \ref{thm:lln} under the hypothesis that $\e(N^2) < \infty$.
\end{proof}

\begin{rema}\label{mogulskii}
 Mogulskii's  theorem (Theorem 5.2.1 of Dembo
and Zeitouni \cite{MR2571413}) implies that  $$ \p \Big( S_{ j} > a j, S_1>0, ..., S_{j} >0 \Big) = e^{- j K(a)  +o(j)}, $$
 with  \begin{gather*} K(a) = \inf\ens{\int_0^1 \psi^*(\dot{f}(t))\, dt,  f  \in \Arond}\,,\\  \Arond=\ens{\phi
   \text{absolutely continuous},  f (0)=0, f(1)=a, f(s)>0 \forall   s\in(0,1)} .  \end{gather*}
 Let us check that $K(a)= \psi^*(a)$. In fact,   since the function $f(t) = a t$ is in $\Arond$, we have $K(a) \le \int_0^1 \psi^*(a)\, dt = \psi^*(a)$. On the other hand, the function $\psi^*$ is convex, therefore, by Jensen's inequality, if $\phi\in\Arond$, $$ \int_0^1 \psi^*(\dot{f}(t))\, dt \ge \psi^*\etp{\int_0^1 \dot{f}(t) dt }=\psi^*(f(1) - f(0)) = \psi^*(a)\,.$$ We can thus conclude that $K(a)=\psi^*(a)$.
\end{rema}

\subsection{Proof of the lower bound, without  the hypothesis $\e(N^2)<\infty$.} \label{s:lbn1}

The proof relies on a coupling for  the general $N$ with mean $m$:  Let $N^{(L)}:= \min (N, L)$ with a sufficiently large integer $L$ such that $m_L:= \e( N^{(L)}) $ satisfies  $m_L (1-\qesc)>1$ (this is possible since $m_L \to m$).  Consider a new CBRW  $(X^{(L)}_u, \vert u \vert \ge0)$ with $N^{(L)}$ as the number of offsprings and  the same random walk $(S_n)$ as the displacements, i.e. on each step of branching at $0$ we keep at most $L$-children and their displacements in the original CBRW. The associated maximum at generation $n$ is denoted by $M^{(L)}_n$. Then by construction $$ M_n \ge M^{(L)}_n, \qquad a.s.$$ By the lower bound for $M_n^{(L)}$ established in Section \ref{s:lbn2},  if we denote by $$ \Srond_L := \left \{ \omega: \limsup_{ n \to \infty} \{ u : \vert u \vert= n, X^{(L)}_u=0\} \not = \emptyset\right\},$$ 

\noindent then a.s. on $\Srond_L$, $$ \liminf_{ n \to \infty} { M^{(L)}_n \over n } \ge \alpha_L, $$ with $ \alpha_L= { \psi(t_0(L))  \over t_0(L)}$, and where $t_0(L)$ is defined in the same way as $t_0$ in \eqref{hyp1}  and \eqref{def:malthusian} by replacing $m$ by $m_L$.   We remark that by continuity  such solution $t_0(L)$ exists for all sufficiently large $L$, say $L\ge L_0$.  Moreover $\alpha_L \to \alpha$ as $L \to \infty$,  and $\Srond_L \subset \Srond_{L+1}$ for any $L\ge1$. Then on the set $\widetilde \Srond:= \cup_{L\ge 1} \Srond_L$, a.s. $ \liminf_{ n \to \infty} { M^{(L)}_n \over n } \ge \alpha. $ This will yield the lower bound in Theorem \ref{thm:lln} once   we have checked the equality:  \begin{equation} \label{eq60} \Srond= \widetilde \Srond, \qquad a.s.. \end{equation}

Let us  check  \eqref{eq60} in the same way as in the proof of Lemma \ref{L:n2}. Plainly $\widetilde \Srond \subset \Srond$. To prove the reverse inclusion, we remark at first that    by Lemma \ref{L:n2}, $\Srond_L$ equals a.s. the non-zero set of the corresponding limit of the fundamental martingale (which is bounded in $L^2$), hence $\Srond_L \not= \emptyset$ for all large $L$. Consequently  $\widetilde \Srond \not= \emptyset$. 

Let $t:= \p(\Srond^c)$ and $\tilde t := \p( \widetilde \Srond^c)$. Then $t \le \tilde t < 1$.  As in the proof of Lemma \ref{L:n2}, by conditioning on the number of offsprings $N$, we obtain that $$ \tilde t= \sum_{k=0}^\infty \prob{N=k} \sum_{j=0}^k C_k^j \qesc^{k-j} (1-  \qesc)^j (\tilde t)^j = f( \qesc + \tilde t ( 1- \qesc)),$$  

\noindent with $f(x)= \E(x^N)$. The constant $t$ satisfies the same equation as $\tilde t$ and we have already proved in the proof of Lemma \ref{L:n2} the uniqueness of solutions  in $[0, 1)$. Hence $t= \tilde t$ and \eqref{eq60} follows. This completes the proof of the lower bound in Theorem \ref{thm:lln}. $\Box$

\section{Refining the Convergence : proof of Theorem~\ref{thm_deuxsec:introduction}.}\label{sec:refining-convergence}

The key of the proof of Theorem~\ref{thm_deuxsec:introduction} is the
following double limit of   Proposition \ref{P:2}. Then we shall prove its uniform version (uniformly on the starting point of the system) in Proposition \ref{P:3}, from which   Theorem~\ref{thm_deuxsec:introduction}  follows easily (see Section \ref{ptm2}).

\begin{prop} \label{P:2} Under the assumptions in
  Theorem~\ref{thm_deuxsec:introduction}, there exists a positive constant $c_*>0$ such that $$ \limsup_{z \to \infty} \limsup_{n \to \infty} \left\vert e^{ t_0 z} e^{- t_0 \{ \alpha n +z\}} \p \Big( M_n > \alpha n +z \Big) - c_* \right \vert =0, $$ where as before $\alpha:= {\psi(t_0)\over t_0}$ and $\{ \alpha n +z\}\in [0, 1)$ denotes the fractional  part of  $\alpha n +z$. 
\end{prop}

The value of $c_*$ is given in \eqref{defc*} by   $c_* = {    e^{-t_0} \over (1-e^{-t_0}) \widetilde \E(H_1)}$ and $\widetilde \e(H_1)$ is
given in equation~\eqref{eq:wald}.  We also mention that we can not replace $M_n > \alpha n + z$ by $M_n \ge \alpha n+z$ in the above Proposition, since $M_n$ is integer-valued. 

 The proof of Proposition \ref{P:2} is divided into the upper and lower bounds, proved respectively in Section \ref{ubp2} and Section \ref{lbp2}.

\subsection{Upper bound in Proposition \ref{P:2} }\label{ubp2}

Recall that  $\alpha:= {\psi(t_0)\over t_0}$ is  the velocity of $M_n$.   We prove the following  upper bound:  for all $z \in \R$, 
$$ \limsup_{n \to \infty}  e^{-  t_0 \{ \alpha n +z\}} \p \Big( M_n > \alpha n  + z\Big) \le  c_*\, e^{ - t_0 z}.$$

Let us start   from   $ \p \Big( M_n > \alpha n  + z\Big)= \p\Big( \exists  \vert u \vert = n: X_u > \alpha n + z\Big). $ 
For any $n\ge 1$ and any $\vert u \vert =n$,  denote by $u_0=\emptyset< u_1<...<u_n=u$ the shortest
path relating $\emptyset$ to $u$ such that $\vert u_k\vert =k$ for any $k \le
n$.  For $\vert u \vert =n$ with $X_u > \alpha n +z>0$ (as $n$ is
large), there exists some $k<n$ such that $X_{u_{k }}=0$ and $X_{u_j}
>0$ for all $k < j \le n$.  Therefore \begin{equation}\label{defbk} \ens{M_n > \alpha n  + z} = \bigcup_{0\le k\le n-1} B_k \end{equation}

\noindent with
 \begin{align*}  B_k  &:= \bigcup_{\vert v\vert =k} A_v(k,
  n)\,, \qquad\text{and}\\
 A_v(k, n) &:= \left\{ \exists \vert u \vert =n: v= u_k, X_v=0,
   X_{u_j} >0, \forall k< j \le n, \, X_{u_{n-k}} > \alpha n + z\right\} \,.
\end{align*}

Denote as before by $\eta_n(x)$ the number of particles at $x$ at time $n$.  Then, conditioning on ${\cal F}_k$, $B_k$ is an union of $\eta_k(0)$ i.i.d. events,  and each event holds with probability  $$ p(k, n):= \p \Big( \exists \vert u \vert = n- k, X_{u_1}>0, ..., X_{u_{n-k}} >0, X_u > \alpha n +z \Big).$$

It is easy to compute $p(k, n)$: by conditioning on the number of offspring $N=l$, $p(k, n)$ is the  probability that among these $l$ particles in the first generation there exists at least one particle which remains positive up to generation $n-k$ and lives in $(\alpha n + z, \infty)$ at $(n-k)$-th generation. It follows  that  \begin{equation}\label{defpkn} p(k, n) =  \sum_{l=0}^\infty \p (N=l)   \Big( 1- ( 1- q(k, n))^l\Big) = 1- f( 1- q(k,n)), \end{equation}  where   $f(x):= \E(x^N)$ is the generating function of $N$ and $q(k, n)$ is defined as follows: $$ q(k, n):= \p\Big( S_1 >0, ..., S_{n-k}>0, S_{n-k} > \alpha n + z\Big)  .$$

\noindent   Let $\varepsilon>0$ be small.   By Proposition~\ref{prop:renewal:etanx} (with $d=1$),
$\lim_{n \to \infty} e^{-r   n } \esp{\eta_{   n}(0)} = c_0= {1\over m} $.
It follows that  for any $n>k \ge k_0\equiv  k_0(\varepsilon)$,  $$ \p \Big( B_k\Big) \le \e \Big( \eta_k(0) p(k, n)\Big) \le  (c_0 + \varepsilon) e^{ r k} \, p(k, n) .$$

Hence  for any $ n >k_0$,  \begin{equation}\label{pm10} \p \Big( M_n > \alpha n  + z\Big) \le \sum_{k=0}^{n-1} \p(B_k ) \le    (c_0 + \varepsilon)    \sum_{k=k_0}^{n-1}   e^{ r k} \, p(k, n) + C_{k_0}  \,    \sum_{k=1}^{k_0-1}    p(k, n), \end{equation}

\noindent where   $C_{k_0}:= \max_{1\le k \le k_0}\e ( \eta_k(0))$. 
Recalling $f'(1)=m$ and \eqref{defpkn}, we deduce from  the convexity of $f $  that  for all $k < n$,   \begin{equation} \label{boundpkn} f'( 1- q(k, n) )  \, q(k, n) \le p(k, n) \le   m \, q(k, n).\end{equation}

  It is easy to see 
that the    sum $ \sum_{k=1}^{k_0-1}$  in \eqref{pm10} is negligible as $n \to
\infty$. In fact, for any $1\le k \le k_0$, $q(k, n) \le \p( S_{n-k-1} > \alpha n +z)$. But $  \E(S_1)= \psi'(0)< \alpha= {\psi(t_0) \over t_0} $ by the (strict) convexity of $\psi$. Then $ p(k, n) \le m q(k, n)\to 0$ as $n \to \infty$ (exponentially fast by the large deviation principle).

To estimate the probability $q(k, n)$ for $k_0 \le k < n$, we
introduce  a new
probability $${ d \widetilde \p \over d \p}\Big\vert_{\sigma\{S_0, ...,
  S_n\}}= e^{t_0 S_n - n \psi(t_0)}.$$ Under $\widetilde \p$, $S_1$
has the mean $\psi'(t_0)>0$.   Therefore for $1\le k
\le n$ and for all $z\ge 0$, 
\begin{eqnarray*} q(k, n) 
&=& \p\Big( S_1 >0, ..., S_{n-k }>0, S_{n-k } > \alpha n + z\Big)  
\\&= &    \widetilde \e \Big( e^{ - t_0 S_{n-k} + (n-k) \psi(t_0)} 1_{( S_j > 0, \forall j \le n-k, S_{n-k } > \alpha n +z)}\Big) \nonumber \\
	&=&    e^{ - r k} \widetilde \p \Big( {\bf e(t_0)} \ge S_{ n-k} - \alpha n ,   S_j > 0, \forall j \le n-k, S_{n-k } > \alpha n +z\Big),    \end{eqnarray*}

\noindent where ${\bf e(t_0)}$ denotes an independent exponential random variable with parameter $t_0$ and we also used the fact that $\alpha= {\psi(t_0)\over t_0}$ and $r=\psi(t_0) $.  Plainly in the event of the above probability term, ${\bf e(t_0)}$ must be bigger than $z$. Thanks to the loss of memory property of ${\bf e(t_0)}$, we get that for $1\le k
\le n$ and for all $z\ge 0$,  \begin{equation}\label{qkn} e^{r k} q(k, n)  = e^{ - t_0 z} \widetilde \p \Big(    S_j > 0, \forall j \le n-k, \alpha n + z < S_{n-k } \le  \alpha n +z+ {\bf e(t_0)}\Big). 
\end{equation}

\noindent 
Summing \eqref{qkn} over $0\le k \le n-1$ and letting $i=n-k$, we obtain that     \begin{eqnarray}  &&   \sum_{k=0}^{n-1}   e^{ r  k} \, q(k , n) \nonumber 
	 \\ &=&  e^{-  t_0 z }   \sum_{i=1}^{n }    \widetilde \p \Big(    S_j > 0, \forall j \le i ,  \alpha n + z < S_{i } \le  \alpha n +z+ {\bf e(t_0)}\Big) \nonumber \\
	&= &      e^{-  t_0 z} \left( \widetilde \E  \Big( U(   \alpha n+ z  ,  \alpha n +z+{\bf e(t_0)} ]   \Big)  -s_n\right),  \label{sumpkn1}   \end{eqnarray}

\noindent where for any $x <y $, \begin{equation} \label{defuxy} U(y):= \sum_{k=1}^\infty \widetilde \p\Big(  S_j> 0 , \forall 1\le j \le k, \,   S_k \le y \Big) , \quad U(x, y]:= U(y) - U(x),\end{equation}  and  \begin{equation}\label{defsn5} s_n:=  \sum_{k=n}^{\infty }    \widetilde \p \Big(    S_j > 0, \forall j \le k,  \alpha n + z < S_{k } \le  \alpha n +z+ {\bf e(t_0)}\Big) . \end{equation}

Under $\widetilde \p$, $S_j  $ is a random walk
with positive mean.  Define by $T_0:=0$, $T_j:=\inf\{ i> T_{j-1}: S_i > S_{ T_{j-1}}\}$  and $H_j:= S_{T_j}$ for  $j\ge 1$.  Then $ 0< T_1 <... < T_j< ...$
and $ 0 < H_1 < \cdots < H_j < \cdots $  are the strict ladder epochs
and ladder heights of the random walk $S$ (under $\widetilde \p$).  The duality lemma  says that for any $y >0$, $$ U(y)= \sum_{l=1}^\infty \widetilde \p \Big( H_l \le y \Big). $$ 

Since $\widetilde\e\etc{S_1^2}< +\infty$, $\widetilde\e (H_1) <
  \infty$ and we have the Wald identity (see \cite{MR0210154} Feller Volume II, Chapter
  XVIII, Theorem 1)
  \begin{equation}
    \label{eq:wald}
 \widetilde\e (H_1) =\widetilde\e (S_1) \widetilde\e (T_1)\,.
  \end{equation}

We are going to apply the   renewal theorem (see \cite{MR0210154} Feller, pp. 360)  to $U$ and prove that there exists some constant $c_H >0$ such that \begin{equation}\label{limituxy}  \lim_{x \to \infty} e^{ - t_0 \{ x\}} \,   \widetilde \E\Big(    U(   x , x + {\bf e(t_0)} ]\Big)       = c_H. 
\end{equation}

To check \eqref{limituxy}, we remark that the span of $H_1$ equals $1$ (because $S$ is aperiodic). By the renewal theorem, for any $j\ge1$, $U(x, x+ j ] \to { j  \over  \widetilde \e(H_1)}$ as $x \to \infty$.  Moreover there exists some constant $C>0$ such that for all $y>x\ge 0$,  $U(x, y] \le C (1+y-x)$. Let $x>0$. Observe that   almost surely, $$ U(x, x + {\bf e(t_0)}]= U( \lfloor x\rfloor, x + {\bf e(t_0)}]= \sum_{j=1}^\infty 1_{ (j < \{x\} + {\bf e(t_0)} < j+1)} U( \lfloor x\rfloor, \lfloor x \rfloor +j] . $$

Taking expectation gives that $$ \widetilde \E\Big(    U(   x , x + {\bf e(t_0)} ]\Big)    =  \sum_{j=1}^\infty e^{ - t_0 (j- \{x \})} (1-e^{-t_0}) U( \lfloor x\rfloor, \lfloor x \rfloor +j] , $$

\noindent which proves \eqref{limituxy} after an application of the dominated convergence theorem, with \begin{equation} \label{defch} c_H:=  \sum_{j=1}^\infty e^{ - t_0  j  } (1-e^{-t_0}) { j \over \widetilde \E(H_1)} = { e^{- 2 t_0} \over (1-e^{-t_0}) \widetilde \E(H_1)} .\end{equation}    

Now we prove that $s_n \to 0$, where $s_n$ is defined in \eqref{defsn5}.  Remark that $\widetilde E(S_1) = \psi'(t_0) > \alpha:= { \psi(t_0) \over t_0} $ by convexity. Pick up some small positive constant $\delta <( \psi'(t_0)  - \alpha)/2$. There exists some sufficiently small constant $b \in (0, t_0)$ such that $ \widetilde \E e^{- b S_1} \le e^{- b ( \psi'(t_0) - \delta)}$. Then by Chebychev's inequality, for any $t>0$ and $k\ge n$, $\widetilde \p ( S_k \le z + \alpha n +t ) \le e^{ b z + b t} e^{ b \alpha n}   \widetilde \E e^{-b S_k} \le e^{ b z + b t} e^{- \delta b k} $. It follows that  $$ s_n \le \sum_{k=n}^\infty e^{b z} e^{- \delta b k} \widetilde \E( e^{ b {\bf e(t_0)}})  = {t_0\over (1-e^{-\delta b}) (t_0-b)} \,  e^{ b z} e^{- \delta b  n}.$$  

\noindent In particular $s_n \to 0$. This together with \eqref{limituxy}, \eqref{sumpkn1}  yields  that   for any $z \ge 0$,  \begin{equation}\label{limitsumqkn} \lim_{n\to \infty} e^{-t_0\,  \{ \alpha n +z\}} \sum_{k=0}^{n-1} e^{r k} q(k, n) = e^{  - t_0 z} c_H.\end{equation}

Now by using the lower bound of \eqref{boundpkn} and \eqref{qkn}, for any  $   k <n$ and $z\ge 0$, $ p(k, n) \ge f'( 1- e^{-rk}) q(k, n) $ because   $q(k, n) \le e^{- r(k+1)} e^{-t_0 z}\le e^{ -r k} $.   Then for any small $\delta>0$,   there exists some $k_0(\delta)$ such that $f'(1-e^{-rk}) \ge m (1-\delta)$ for all $k \ge k_0$ (recalling $f'(1)=m$). It follows that  for any $k_0 \le k < n$ and $z \ge 0$, $ p(k, n) \ge (1-\delta) m q(k, n)$. On the other hand, $p(k, n) \le m q(k, n)$ for any $k<n$, and $\lim_{ n \to \infty} \sum_{k=0}^{k_0} q(k, n) =0$. This in view of  \eqref{limitsumqkn} implies that for any $z\ge 0$, \begin{equation} \label{limsumpkn} \lim_{n\to \infty} e^{-t_0\,  \{ \alpha n +z\}} \sum_{k=0}^{n-1} e^{r k} p(k, n) = m e^{  - t_0 z} c_H.\end{equation}

Applying the above limit to \eqref{pm10} gives that     for any $z \ge 0 $, $$\limsup_{n \to \infty} e^{ -t_0\, \{ \alpha n +z\}} \p( M_n > \alpha n + z) \le (c_0 + \varepsilon)  \,m \,   c_H\,   e^{- t_0 z},$$

\noindent which implies  the upper bound in Proposition \ref{P:2} by letting $\varepsilon \to 0$ and \begin{equation} \label{defc*} c_*:= c_0 m    c_H\,  = { c_0 \,m\,e^{ -2 t_0} \over (1-e^{-t_0}) \widetilde \E(H_1)} ={  e^{  - 2 t_0} \over (1-e^{-t_0}) \widetilde \E(H_1)}  ,\end{equation} since $c_0= 1/m$ (the period $d=1$) as stated  in Proposition  \ref{P:41}. 
$\Box$

 \begin{rema}
Let us mention an uniform estimate:  for some constant $C>0$,   \begin{equation}\label{uni} \p( M_n > \alpha n + z) \le  C e^{- t_0 z}, \qquad \forall z \in \R, n\ge1. \end{equation} In fact,  there exists some constant $C'>0$ such that $\E(\eta_k(0) ) \le C' e^{ r k}$ for any $k\ge1$, hence by the first inequality in \eqref{pm10}, $\p( M_n > \alpha n+  z) \le C' \sum_{ k=0}^{ n-1} e^{ r k} p(k,n) \le C' m \sum_{ k=0}^{ n-1} e^{ r k} q(k,n). $  Using \eqref{sumpkn1} and the fact that $\exists C^{''}>0$:  $U(x, y] \le C^{''} (1+y-x)$ for all $x <y$, we immediately get \eqref{uni}.  
 \end{rema}

 \begin{rema}\label{R:5}  If the underlying random walk $S$ is of period $d\ge2$, then in \eqref{defbk}, $B_k=\emptyset$ if $k$ is not multiple of $d$ (namely if $d \not \vert \, k$).  Then instead of $\sum_{k=0}^{n-1} e^{r k} p(k, n)$, we have to deal with $\sum_{k=0, d \vert k}^{n-1} e^{r k} p(k, n)$, which in turn leads to the study of $\sum_{k=0, d \vert  k}^{n} e^{r k} q(k, n)$. An equality similar to \eqref{sumpkn1} holds with $U$ replaced by $$ U^{(d, \ell)}(y):= \sum_{k=0}^\infty \widetilde \p\Big(  S_j> 0 , \forall 1\le j \le k d+ \ell, \,   S_{kd + \ell} \le y \Big) ,$$  where $\ell \in \{0, ..., d-1\}$ comes from the rest of division  of $n$ by $d$ [$\ell$ being fixed and we let $n\to \infty$ with $n-1\equiv \ell (\mbox{mod } d)$].  Technically we are not able to prove any  renewal theorem for $U^{(d, \ell)}(y)$   for a general random walk $S$. 
 
 In the particular case when $S$ is a nearest neighbor random walk on
 $\Z$, we can use  parity to handle $U^{(d, \ell)}(y)$.   Considering for instance $\ell=0$ ($d=2$). Thanks to  parity, we have that for any $k \ge1$ and $y >0$, $$ \widetilde \p\Big(  S_j> 0 , \forall 1\le j \le  2 k , \,   S_{2 k} \le y \Big) = \widetilde \p\Big(  S_{2j}> 0 , \forall 1\le j \le    k , \,   S_{2 k} \le y \Big),$$ which implies that $U^{(2,0)}(y)$ is the renewal function for the random walk $(S_{2n})_{n\ge0}$ (under $\widetilde \p$). Then we can apply the standard renewal theorem to $U^{(2,0)}(y)$. The term $U^{(2,1)}(y)$ can be dealt with in the same way. Then  we   get a result similar to Proposition \ref{P:2}  and the forthcoming  Proposition \ref{P:3}, and finally   a modified version of  Theorem \ref{thm_deuxsec:introduction}  for the nearest neighbor random walk. The details are omitted. 
 \end{rema}

\subsection{Lower bound in Proposition \ref{P:2} }\label{lbp2}

  Let $\varepsilon>0$ be small.  Let $\lambda\equiv\lambda(\varepsilon)$ be a large constant whose value will be determined later on. Recall \eqref{defbk}. Consider   $$E_n:= \bigcup_{ k=0}^{n-1}  B'_k,$$
with $B'_k:= B_k  \cap \{  \eta_k(0) \le  \lambda e^{ r k  }  \}:= B_k \cap F_k.$ 
Then by Cauchy-Schwarz' inequality,  \begin{equation}\label{cs}  \p\Big( M_n > \alpha n + z  \Big)     \ge \p\Big( E_n \Big)  \ge { \Big(  \sum_{ 0\le k <n} \p( B'_k)\Big)^2 \over  \sum_{ 0\le  k_1 ,  k_2 < n} \p( B'_{k_1} \cap B'_{k_2})}.  \end{equation}

  Conditioning on ${\cal F}_k$, $B_k$ is an union of $\eta_k(0)$ i.i.d. events, $$ \p\Big( B_k \big \vert { \cal F}_k\Big) =  1- ( 1- p(k, n))^{\eta_k(0)}.$$ 

Let $0 \le k < n$. By (\ref{boundpkn}) and \eqref{qkn}, $ p(k, n) \le m e^{ - r (k+1)} e^{ - t_0 z }$.  On $F_k$, $\eta_k(0) \le \lambda \, e^{  r k }  $ hence $p(k, n) \eta_k(0) \le e^{-r - t_0 z} m  \lambda $.   Therefore  for all $z \ge z_0(\lambda, \varepsilon)$ and    for all $k < n$,    
$$1- ( 1- p(k, n))^{\eta_k(0)} \ge (1-\varepsilon) p(k, n)
\eta_k(0 )\,$$
 hence  $$ \p\Big( B'_k \big \vert { \cal F}_k\Big)  \ge (1-\varepsilon) p(k, n) \eta_k (0) 1_{ F_k}.$$

In particular, $$  \sum_{k=0}^{n-1}  \p( B'_k)  \ge (1-\varepsilon) \sum_{k=0}^{n-1} p(k, n)   \e \Big(  \eta_k (0) 1_{ F_k}\Big).$$ 

 Since $\eta_k(0) e^{-r k}$ is bounded in $L^2$,  we deduce from Proposition \ref{prop:renewal:etanx} that we can choose (and then fix) some large $\lambda$  and some $k_0\equiv k_0(\varepsilon)$ such that $   \e \Big(  \eta_k (0) 1_{ F_k}\Big)  \ge  c_0(1 - \varepsilon)  e^{ r k }$ for all $k \ge k_0$.
   It follows that   for all $n > k_0$, $$    \sum_{k=0}^{n-1}  \p( B'_k)  \ge  c_0 (1- \varepsilon)^2    \sum_{k=k_0}^{n-1} e^{r k} p(k, n)   .  $$

Consequently,  for all $z\ge z_0$ there exists some $ n_0(z, \varepsilon)$ such that for all $n \ge n_0$, \begin{equation} \label{pb'k1} \sum_{k=0}^{n-1}  \p( B'_k)  \ge  c_0\, m \,  (1- \varepsilon)^3    \sum_{k=k_0}^{n-1} e^{r k} q(k, n) \ge  c_*    (1- \varepsilon)^4 e^{ - t_0 z} e^{ t_0\, \{ \alpha n +z\}},    \end{equation}  

\noindent by applying \eqref{limitsumqkn}  [recalling $c_*=  c_0 m     c_H $, $c_0=1/m$ and that for any fixed $k$ $ q(k,n)\to 0$ as $n \to\infty$].  The probability $\p(B_k)$ has already  been   estimated in the proof of upper bound of Proposition \ref{P:2}, see \eqref{pm10} and \eqref{limsumpkn}: for all    $z\ge 0$ and $n \ge n_0(z, \varepsilon)$, \begin{equation} \label{pb'k2} \sum_{k=1}^{n-1}   \p( B'_k) \le  \sum_{k=1}^{n-1} \p( B_k) \le  c_* (1+ \varepsilon) e^{ - t_0 z} e^{t_0\,  \{ \alpha n +z\}} .  \end{equation}

Now we estimate the denominator in \eqref{cs}. 
Let $k_1 < k_2$. On $B_{k_1} \cap B_{k_2}$,   there are at least two different $v \not =v' $ at generation $k_1$ such that  $A_v(k_1, n)$ holds and for $v'$, there exists some descendant  $u $ (denoted by $u > v'$) at generation $ k_2$ such that $A_u(k_2, n)$ holds.  Then, $$ B_{k_1} \cap B_{k_2} \subset \bigcup_{ v \not = v', \vert v\vert=\vert v'\vert=k_1} \Big\{ A_v(k_1, n) \cap \{ \exists   \vert u \vert = k_2, u > v': A_u(k_2, n) \mbox{ holds}\}\Big\}.$$ 

\noindent
Since different particles branch independently,  we get that $$ \p\Big( B_{k_1} \cap B_{k_2} \big \vert {\cal F}_{k_1} \Big) \le \sum_{ v \not = v', \vert v\vert=\vert v'\vert=k_1} p(k_1, n) \,  \e \Big( \sum_{ \vert u \vert = k_2, u > v'} p(k_2, n) \big\vert {\cal F}_{k_1}\Big).$$

\noindent Taking the expectations,  we obtain that for $k_1 < k_2$, $$ \p\Big( B_{k_1} \cap B_{k_2}\Big) \le p(k_1, n) p(k_2, n) \e ( \eta_{k_1}(0) \eta_{k_2}(0) ) \le C\, p(k_1, n) p(k_2, n) e^{ r(k_1  + k_2 )}  ,$$

\noindent by Corollary~\ref{cor:secmomentetan}.  Therefore  for all $z\ge z_0$ and $n>n_0(z, \varepsilon)$,  \begin{eqnarray*} \sum_{0\le k_1, k_2< n } \p\Big( B'_{k_1} \cap B'_{k_2}\Big)& \le & \sum_{k=0}^{n-1}  \p( B'_k)    + C \Big( \sum_{k=0}^{n-1} e^{r k} p(k, n)\Big)^2  \\
	&\le&   c_* (1+\varepsilon )  e^{- t_0 z} e^{t_0\,  \{ \alpha n +z\}} + C'  e^{- 2 t_0 z},\end{eqnarray*} for some numerical constant $C '>0$.  In view of \eqref{cs}, we have     that for all $z\ge z_0$ and $n>n_0(z, \varepsilon)$,  $$ \p\Big( M_n > \alpha n + z\Big) \ge {   c_*^2 (1-\varepsilon)^8    e^{-     t_0 z} e^{ t_0\,  \{ \alpha n +z\}} \over c_* (1+\varepsilon )  + C'  e^{-   t_0 z  }}.$$

	\noindent   It follows  that $$ \liminf_{z \to \infty} \liminf_{n \to \infty} e^{t_0 z - t_0 \{ \alpha n +z\}}  \p\Big( M_n > \alpha n + z\Big) \ge c_*  {(1-\varepsilon)^8 \over 1+\varepsilon}.$$ 

 \noindent  Letting $\varepsilon \to0$, we obtain the  lower bound in    Proposition \ref{P:2}.  The proof of  Proposition \ref{P:2} is complete. $\Box$

\medskip 
Recall  that $\phi(x)$ is defined in  \eqref{defphix} and $\phi(x)>0$ thanks to the aperiodicity. Let us establish  an uniform version of Proposition \ref{P:2}:

\begin{prop} \label{P:3} Under the assumptions in Theorem~\ref{thm_deuxsec:introduction}.  Uniformly on   $x\in \Z $, $$   \limsup_{n \to \infty} \left\vert { e^{ t_0 z} e^{ -t_0\, \{ \alpha n +z\}} \over \phi(x)} \, \p_x \Big( M_n > \alpha n +z \Big) - c_*   \right \vert \to 0\,, $$ as $z \to \infty$. 
\end{prop}

\begin{proof} Assume $x\neq 0$ and let $S^*=  \max_{ 0\le i \le \tau} S_i, $ where $\tau$ is the first return time to $0$.  Then  $$ \p_x \Big( M_n > \alpha n +z \Big)  \le \p_x  ( S^*> \alpha n +z  )  + \sum_{k=1}^n \p_x(\tau=k)  \p\Big( M_{n-k} > \alpha n +z\Big). $$

Let $\varepsilon>0$ be small (in particular $\varepsilon < c_*$).  Let  $\ell $ be some integer whose value will be fixed later on. By  Proposition \ref{P:2},  there exists some $y_0(\varepsilon)>0$ such that for all $ y \ge y_0(\varepsilon)$, there exists some $j_0(y, \varepsilon)$ such that for  all $j \ge j_0(y, \varepsilon)$,   \begin{equation}\label{y0j0}  \big\vert e^{  t_0 y } e^{ -t_0\,  \{ \alpha j +y \}} \p\Big( M_j  > \alpha  j +y \Big) -  c_*  \big\vert < \varepsilon.\end{equation}

Observe that for any $k<n$,  $  \p\Big( M_{n-k} > \alpha n +z\Big ) = \p\Big( M_{n-k} > \alpha (n -k)+z+ \alpha k\Big) $.  We shall apply \eqref{y0j0} to $y= \alpha k+ z$ and $j=n-k$.  Then for all $z \ge y_0(\varepsilon)$, there exists some $j_1(z, \ell)$ such that for all   $1 \le k \le \ell$ and $n \ge j_1(z, \ell)$,   \begin{equation} \label{y0j1} \big\vert e^{  t_0 (z+ \alpha k)} e^{ -t_0\,  \{ \alpha n +z \}} \p\Big( M_{n-k} > \alpha n +z\Big ) - c_*\big \vert < \varepsilon  . \end{equation} 

We stress  that $y_0(\varepsilon)$ does not depend on $\ell$. 
Then  for all $n > j_1(z, \ell)$,  \begin{eqnarray*}  \sum_{k=1}^\ell  \p_x(\tau=k)  \p\Big( M_{n-k} > \alpha n +z\Big) &\le&  (c_*+\varepsilon)  e^{-t_0 z} e^{ t_0\, \{ \alpha n +z \}} \sum_{k=1}^\ell \p_x(\tau=k)  e^{- \alpha t_0 k} 
	\\ &\le& (c_*+\varepsilon)  e^{-t_0 z} e^{t_0\,  \{ \alpha n +z \}} \phi(x), \end{eqnarray*}

\noindent since $\alpha t_0= \psi(t_0)=r $ and $ \phi(x)= \E_x \etc{e^{- r \tau}}.$ For $k> \ell$, we apply \eqref{uni} and get that $$\sum_{k=\ell}^n \p_x(\tau=k)  \p\Big( M_{n-k} > \alpha n +z\Big)  \le \sum_{k=\ell}^n  C e^{ -t_0 ( \alpha k +z)} =C e^{-t_0 z} { e^{- r \ell} \over r}. $$

\noindent It follows that for any $z\ge y_0(\varepsilon)$ and any $x \in \Z$,  \begin{equation}\label{upphix} \limsup_{ n \to \infty} e^{ t_0 z} e^{ -t_0\,  \{ \alpha n +z\}} \p_x \Big( M_n > \alpha n +z \Big)   \le   (c_*+\varepsilon) \phi(x) + C    { e^{1 - r \ell } \over r}. \end{equation}

For the lower bound, we have from \eqref{y0j1} that for any  $z \ge y_0(\varepsilon)$ and  all $n > j_1(z, \ell)$,  \begin{eqnarray*}  \sum_{k=1}^\ell  \p_x(\tau=k)  \p\Big( M_{n-k} > \alpha n +z\Big) &\ge&  (c_*- \varepsilon)  e^{-t_0 z} e^{ t_0\,  \{ \alpha n +z \}} \sum_{k=1}^\ell \p_x(\tau=k)  e^{- \alpha t_0 k} 
	\\ &= & (c_*- \varepsilon)  e^{-t_0 z} e^{ t_0\, \{ \alpha n +z \}} \E_x \etc{e^{- r \tau} 1_{(\tau \le \ell)}}. \end{eqnarray*}

\noindent Hence  for any $z\ge y_0(\varepsilon)$ and any $x \in \Z$, $$ \liminf_{ n \to \infty} e^{ t_0 z} e^{ -t_0\, \{ \alpha n +z\}} \p_x \Big( M_n > \alpha n +z \Big)   \ge   (c_*-\varepsilon)  \E_x \etc{e^{- r \tau} 1_{(\tau \le \ell)}}. $$

Letting $\ell \to \infty$ in the above $\liminf$ inequality and in \eqref{upphix} gives that for any $z \ge y_0(\varepsilon)$ and uniformly  for all $x \in \Z$, \begin{equation}\label{epsilonphix} \limsup_{n \to \infty} \left\vert e^{ t_0 z} e^{ -t_0\, \{ \alpha n +z\}} \p_x \Big( M_n > \alpha n +z \Big) - c_*  \phi(x)\right \vert   \le \varepsilon \phi(x) , \end{equation}
proving Proposition \ref{P:3} since $\varepsilon$ can be arbitrarily small. 
\end{proof}

\subsection{Proof of Theorem~\ref{thm_deuxsec:introduction}}\label{ptm2}

The part \eqref{eq:cvferchetcond} of
Theorem~\ref{thm_deuxsec:introduction} was already proved in Lemma \ref{L:n2}. We now prove \eqref{thm2:a}. 

Let $\varepsilon, \delta>0$ be small.  For any $k\ge 1$, there exists some integer $\ell_k=\ell_k(\varepsilon)$ such that  $$\p\Big( \max_{ \vert u \vert = k} \vert X_u\vert  \le \ell_k  \Big) \ge 1- \varepsilon.$$ 

Recalling the martingale $\Lambda_n$ defined in Proposition \ref{P:41}. Since a.s. $\Lambda_n \to \Lambda_\infty$, there exists some 
$k_1=k_1(\varepsilon, \delta)$ such that for any $k \ge k_1$, $$ \p \Big( (1-\delta) \Lambda_\infty \le \Lambda_k \le (1+\delta) \Lambda_\infty \Big) \ge 1- \varepsilon.$$

 By \eqref{epsilonphix}, there exists some $z_0(\delta)$ such that for all $z \ge z_0(\delta)$  and  for all $x \in \Z$, there exists some $n_0(z,   x, \delta)$ such that for all $ j\ge n_0(z,   x, \delta)$,  \begin{equation}\label{mpmp} \left\vert e^{ t_0 z} e^{  -t_0\, \{ \alpha j +z\}} \p_x \Big( M_j > \alpha j +z \Big) - c_*  \phi(x)\right \vert \le \delta \phi(x).\end{equation}

Elementarily there exists some $s_0(\delta)>0$ such that  $1- s \ge e^{- (1+\delta) s}$ for all $0\le s < s_0(\delta)$. Let $k_2=k_2(\delta, y)$ be some integer satisfying    $ (c_* +\delta) e^{-   t_0 (\alpha k_2+y-1)} < s_0(\delta)$. 
Define $k: =   k_1 + k_2+ \lfloor {  z_0(\delta)  \over \alpha} \rfloor +1 $. Let $n_1:= \max_{ x \in \Z, \vert x \vert \le \ell_k}n_0 (z, x, \delta)+ k$. Considering $n\ge n_1$.  Conditioning on ${\cal F}_k$ and on the set $\{\max_{ \vert u \vert = k} \vert X_u\vert  \le \ell_k\}$, the   particles  in the $k$-th generation move independently, hence for any $n >n_1$,  \begin{equation}\label{pm20} \p\Big( M_n > \alpha n + y \big\vert {\cal F}_k \Big)   =    1-  \prod_{x \in \Z, \vert x \vert \le L k} \p_x\Big( M_{ n-k} \le \alpha n + y\Big)^{\eta_k(x)} .\end{equation}

\noindent Applying \eqref{mpmp} to $j= n-k $, $ z= \alpha k +y $  yields  that  for any $\vert x \vert \le  \ell_k$ (and $x \in \Z$),   $$  (c_* -\delta) \phi(x) e^{ -t_0 (\alpha k +y) +t_0 \{ \alpha n+y\}}\le   \p_x\Big( M_{  n-k} > \alpha n + y\Big) \le  (c_* +\delta) \phi(x) e^{ -t_0 (\alpha k +y) +t_0 \{ \alpha n+y\}}.$$

Since $1- s \ge e^{- (1+\delta) s}$ for all $0\le s < s_0(\delta)$, we deduce from \eqref{pm20}   that on the set $\{\max_{ \vert u \vert = k} \vert X_u\vert  \le  \ell_k\}$,  \begin{eqnarray} \label{mpmp2}  && \p\Big( M_n > \alpha n + y \big\vert {\cal F}_k \Big)  \nonumber\\ &\le  &1 - \exp\Big( - \sum_{x \in \Z, \vert x \vert \le \ell_k} (c_*+  \delta) (1+\delta) \phi(x) \eta_k(x)  e^{ -t_0 (\alpha k+  y)} e^{ t_0 \{ \alpha n+y\}}  \Big)  \nonumber
	\\ &=& 1 - \exp\Big( -   (c_*+  \delta) (1+\delta) \Lambda_k  e^{ -t_0  y } e^{ t_0 \{ \alpha n+y\}}  \Big) . \end{eqnarray}

	Then by taking the expectation, we get \begin{eqnarray*}  && \p\Big( M_n > \alpha n + y  \Big)  \nonumber\\ &\le  &   \E \left( 1-  \exp\Big( -   (c_*+  \delta) (1+\delta) \Lambda_k  e^{ -t_0  y} e^{ t_0 \{ \alpha n+y\}}  \Big) \right) + \p\Big( \max_{ \vert u \vert = k} \vert X_u\vert  >  \ell_k \Big) \\
	&\le&  \E \left( 1-  \exp\Big( -   (c_*+  \delta) (1+\delta)^2 \Lambda_\infty  e^{ -t_0  y} e^{ t_0 \{ \alpha n+y\}}  \Big) \right) + 2 \varepsilon,	\end{eqnarray*}
	
	\noindent where the factor $2$ in $2 \varepsilon$ comes from   $\Lambda_k$ which is replaced by  $(1+\delta)\Lambda_\infty$. Since $\varepsilon$ and $\delta$ can be arbitrarily small, we get the upper bound in \eqref{thm2:a}. The lower bound in \eqref{thm2:a} can be proved in the same way.

	Finally, let $y\in \Z$. Observe  that for any $n_j \ge1$,  $\p( M_{n_j} - \lfloor \alpha n_j\rfloor \ge y +1) = \p( M_{n_j} - \lfloor \alpha n_j\rfloor > y + \{ \alpha n_j\})= \p( M_{n_j} - \alpha n_j  > y  ) $. We apply \eqref{thm2:a} to $y$ and $y-1$, \eqref{thm2:b} follows immediately. This completes the proof of 	 Theorem 	\ref{thm_deuxsec:introduction}.  $\Box$.

\section{Extension to multiple
  catalysts branching random walk (MCBRW)}\label{sec:extens-mult-catalyst}

Recall Section \ref{subsec1} for the definition of MCBRW. Let us assume that the set of catalysts $\Crond$ is a finite subset of $ \Z$.   By forgetting/erasing
the time spent between the catalysts, we obtain an  underlying Galton-Watson process which is multitype  with the
moment matrix
\begin{align*}
  M_{xy} &:= \text{mean number of particles born at $x$ that reach
    site $y$} \\
&= m_1(x) \PP_x\etp{\tau= \tau_y, \tau < \infty} &(x,y\in\Crond)\,,
\end{align*}
where $m_1(x)=\esp{N_x}$ is the mean offspring at site $x$,
$\tau_y:=\inf\ens{n\ge 1 : S_n=y}$ is the first return time at $y$,
and $\tau=\tau_\Crond=\inf_{y\in\Crond}\tau_y$ is the first return
time to $\Crond$.

We assume to be in the supercritical regime, that is $\rho >1$, where
$\rho$ is the maximal eigenvalue of matrix $M$, which by assumption is
irreducible. We let $\rho^{(r)}$ be the maximum eigenvalue of the
matrix
$$ M^{(r)}_{xy} := m_1(x) \esperance{x}{e^{-r \tau} \un{\tau = \tau_y, \tau < \infty}}
\quad(x,y\in\Crond).$$

The function $r\to \rho^{(r)}$ is continuous, strictly decreasing,
$C^\infty$ on $(0,+\infty)$, $\rho^{(0)}=\rho >1$ and $\lim_{r\to
  +\infty} \rho^{(r)}=0$ since $M^{(r)}_{xy} \le m_1(x)
e^{-r}$. Therefore there exists a unique $r>0$, a \emph{Malthusian
parameter}, such that 
$ \rho^{(r)}=1$. We shall fix this value of $r$ in the sequel.

Let $v=v^{(r)}$ be a right eigenvector of $M^{(r)}$ associated to
$\rho^{(r)}=1$:  For any $x \in \Crond$, $v(x)>0$ and 
$$   \quad
v(x) = \sum_{a \in\Crond}  m_1(x) \esperance{x}{e^{-r \tau} \un{\tau
    = \tau_a, \tau < \infty}} v(a)\quad (x\in\Crond)\,.$$

Let us denote by $p(x,y)=\esperance{x}{S_1=y}$ and $Pf(x) = \sum_y
p(x,y) f(y)$ the random walk kernel and semigroup. Let us consider the
hitting times
$$ T_x :=\inf\ens{ n\ge 0 : S_n=x}\,,\quad  T_\Crond =
\inf_{x\in\Crond} T_x =\inf\ens{n\ge 0 : S_n \in\Crond}\,.$$

\begin{lemm}
  The function 
$$ \phi(x)  := \sum_{a\in\Crond} v(a)  \esperance{x}{e^{-r T_\Crond }
  \un{T_\Crond =T_a, T_\Crond <\infty}}$$
 is a solution of 
$$ P\phi(x) = e^r \phi(x) \etp{\unsur{m_1(x)} \un{x\in\Crond} + \un{x\notin\Crond}}.$$
\end{lemm}
\begin{proof} The proof is similar to that of Proposition \ref{P:41} by using the Markov property of the random walk. The details are omitted.
\end{proof}

We are now ready to introduce the \emph{fundamental martingale}.
\begin{lemm}
(1)  For the CBRW process with multiple catalysts, the process
$$ \Lambda_n := e^{-rn} \sum_{\valabs{u}=n} \phi(X_u)$$
is a martingale.

(2) For the random walk, the process 
$$ \Delta_n := e^{-rn} \phi(S_n) \prod_{x\in\Crond} m_1(x)^{L^x_{n-1}}$$
 is a martingale where $L^x_{n-1} = \sum_{0\le k\le n-1} \un{S_k=x}$
 is the local time at level $x$ at time $n-1$.

(3) If $N$ has finite variance, then the process $\Lambda_n$ is bounded in $L^2$ and therefore a uniformly
integrable martingale.
\end{lemm}
\begin{proof}  Based on the many-to-one formula, the parts (1) and (2) can be proved in the same way as in Proposition \ref{P:41}. Let us only give the details of the proof of  
(3).  To compute  the second moment, we use the many to two formula \eqref{mtotogbrw}  of
Section~\ref{sec:many-few-formulas}
\begin{align*}
  \esp{\Lambda_n^2} &= e^{-2 rn} \esp{ \sum_{\valabs{u}=\valabs{v}=n}
    \phi(X_u) \phi(X_v)} \\
&= e^{-2 rn} \Q\etc{\phi(S^1_n) \phi(S^2_n) \prod_{0\le k< T\wedge n} m_2(S^1_k)
    \prod_{T\wedge n \le k < n}m_1(S^1_k) m_1(S^2_k)}\,.
\end{align*}
Recall \eqref{decouple}. We have that
\begin{align*}
   \esp{\Lambda_n^2} 
&= e^{-2rn}  \Q\etc{ \phi(S_n)^2 \prod_{x\in\Crond}
  m_1(x)^{L^x_{n-1}}} \\
&+ e^{-2rn}   \sum_{1\le k\le n-1} \Q\etc{ \prod_{0\le l\le k-2}
  \frac{m_1(S_l)}{m_2(S_l)}  (1-
  \frac{m_1(S_{k-1})}{m_2(S_{k-1})})
  \esperance{S_{k-1}}{\Delta_{n-(k-1)}}^2 e^{2 r(n-(k-1))} }.
\end{align*}
Observe that since $0\le \phi\le 1$ we have $0\le \phi(x)^2\le \phi(x)$,
and combine it with $\esperance{x}{\Delta_p} = \phi(x) m_1(x)\le C$ and ${m_1(x) \over m_2(x)} \le 1$ to
obtain the upper bound $$ \esp{\Lambda_n^2} \le 1 + C^2 \sum_{1\le k\le n-1} e^{- 2 r(k-1)} \le C' < \infty,$$ 
which completes  the proof of this Lemma. \end{proof}

 \medskip

We are now able to give an explanation of the supercritical regime
assumption of the introduction.
\begin{lemm}\label{lem:extens-mult-catalyst-1}
  When there is only one catalyst at the origin, the supercritical
  regime is $m(1-\qesc) >1$.
\end{lemm}
\begin{proof}
  Indeed, $M$ is then a one dimensional matrix and 
 $ \rho = M_{00}= m\prob{\tau < +\infty} = m (1-\qesc)\,.$ 
\end{proof}

We end this section by stating the law of large numbers. Intuitively,
if $\overline c$ is the rightmost catalyst, the maximal position at time $n$
comes from particles born at location  $\overline c$.
\begin{prop}[Law of large numbers] Assume the supercritical regime and
  \eqref{hyp1}. Then, 
  on the set of non extinction $\Srond$ we have
$$ \lim_{n\to +\infty} \frac{M_n}{n} = \frac{r}{t_0} , \quad a.s.,$$
with  $r$ the Malthusian parameter defined by $\rho^{(r)}=1$ and
$t_0>0$ such that $\psi(t_0)=r$.
\end{prop}
\begin{proof}
   First observe that the heuristics do not change at all since by
  applying the optional stopping theorem to the martingale $e^{t_0 S_n
    - n r}$ to the time $T$, we obtain that for $x>\overline c$
$$e^{t_0 x} = e^{t_0 \overline c} \esperance{x}{e^{-rT}}\,,$$
and thus 
 
$$\phi(x)=v(\overline c) \esperance{x}{e^{-r
    T_c}}=v(\overline c) e^{t_0(x-\overline c)}\,,$$
and we approximate the expected number of particles above level $an$ in
the same way, and hence obtain the same guess for the asymptotics.

Furthermore, the proofs are \emph{mutatis mutandis} the same as the
one given in section~\ref{sec:law-large-numbers}. The only difference would come from the use of renewal theorems: we get a system of renewal equations, e.g., for $(\E_{a} ( e^{ \theta S_n } \prod_{b \in \Crond} m_1(b)^{ L^b_{n-1}}))_{a, b \in \Crond}$ as $n \to \infty$, which can be dealt with an application of a matrix version of renewal theorems (see \cite{crump, deaporta}). We feel free to omit the details. \end{proof}

{\noindent \bf \large Acknowledgements.}  We are very grateful to two anonymous referees for their  careful readings and helpful  comments on the first version of this paper. 
\bibliographystyle{amsplain}

\def\cprime{$'$}

\end{document}